\documentclass[11pt]{article}
\usepackage[english]{babel}
\usepackage[latin1]{inputenc}
\usepackage{amsmath}
\usepackage{amssymb}
\usepackage{amsfonts}
\usepackage{amsthm}
\usepackage{upgreek}
\usepackage{stmaryrd}

\renewcommand{\a}{\mathfrak{a}}
\newcommand{\n}{\mathfrak{n}}

\renewcommand{\v}{\mathfrak{v}}
\newcommand{\w}{\mathfrak{w}}
\newcommand{\x}{{\rm x}}
\newcommand{\C}{\mathbb{C}}
\newcommand{\N}{\mathbb{N}}

\newcommand{\R}{\mathbb{R}}

\newcommand{\boC}{\mathcal{C}}
\newcommand{\boE}{\mathcal{E}}
\newcommand{\boF}{\mathcal{F}}

\newcommand{\boM}{\mathcal{M}}
\newcommand{\boN}{\mathcal{N}}
\newcommand{\boO}{\mathcal{O}}
\newcommand{\boP}{\mathcal{P}}
\newcommand{\boR}{\mathcal{R}}

\newcommand{\eps}{\varepsilon}
\newcommand{\ch}{{\rm ch}}

\renewcommand{\Im}{{\rm Im}}

\newtheorem{cor}{Corollary}
\newtheorem{lemma}{Lemma}
\newtheorem{prop}{Proposition}
\newtheorem{step}{Step}
\newtheorem{theorem}{Theorem}
\theoremstyle{definition}
\newtheorem*{merci}{Acknowledgements}
\newtheorem{remark}{Remark}

\begin{document}

\title{On the Korteweg-de Vries long-wave approximation of the Gross-Pitaevskii equation I}
\author{
\renewcommand{\thefootnote}{\arabic{footnote}}
Fabrice B\'ethuel \footnotemark[1], Philippe Gravejat \footnotemark[2], Jean-Claude Saut \footnotemark[3], Didier Smets \footnotemark[4]}
\footnotetext[1]{UPMC, Universit\'e Paris 06, UMR 7598, Laboratoire Jacques-Louis Lions, F-75005, Paris, France. E-mail: bethuel@ann.jussieu.fr}
\footnotetext[2]{Centre de Recherche en Math\'ematiques de la D\'ecision, Universit\'e Paris Dauphine, Place du Mar\'echal De Lattre De Tassigny, 75775 Paris Cedex 16, France, and \'Ecole Normale Sup\'erieure, DMA, UMR 8553, F-75005, Paris, France. E-mail: gravejat@ceremade.dauphine.fr}
\footnotetext[3]{Laboratoire de Math\'ematiques, Universit\'e Paris Sud and CNRS UMR 8628, B\^atiment 425, 91405 Orsay Cedex, France. E-mail: Jean-Claude.Saut@math.u-psud.fr}
\footnotetext[4]{UPMC, Universit\'e Paris 06, UMR 7598, Laboratoire Jacques-Louis Lions, F-75005, Paris, France, and \'Ecole Normale Sup\'erieure, DMA, UMR 8553, F-75005, Paris, France. E-mail: smets@ann.jussieu.fr}
\maketitle

\begin{abstract}
The fact that the Korteweg-de-Vries equation offers a good approximation of long-wave solutions of small amplitude to the one-dimensional Gross-Pitaevskii equation was derived several years ago in the physical literature (see e.g. \cite{KuznZak1}). In this paper, we provide a rigorous proof of this fact, and compute a precise estimate for the error term. Our proof relies on the integrability of both the equations. In particular, we give a relation between the invariants of the two equations, which, we hope, is of independent interest.
\end{abstract}

\section{Introduction}

In this paper, we consider the one-dimensional Gross-Pitaevskii equation
\renewcommand{\theequation}{GP}
\begin{equation}
\label{GP}
i \partial_t \Psi + \partial_\x^2 \Psi = \Psi (|\Psi|^2 - 1) \ {\rm on} \ \R \times \R,
\end{equation}
which is a version of the defocusing cubic nonlinear Schr\"odinger equation and appears as a relevant model in various areas of physics: Bose-Einstein condensation, fluid mechanics (see e.g. \cite{GinzPit1,Pitaevs1,Gross1,Coste1}), nonlinear optics (see e.g. \cite{KivsLut1}).

We supplement this equation with the boundary condition at infinity
\renewcommand{\theequation}{\arabic{equation}}
\setcounter{equation}{0}
\begin{equation}
\label{bdinfini}
|\Psi(\x, t)| \to 1, \ {\rm as} \ |\x| \to + \infty.
\end{equation}
This boundary condition is suggested by the formal conservation of the energy (see \eqref{GLE} below), and by the use of the Gross-Pitaevskii equation as a physical model, e.g. for the modelling of ``dark solitons'' in nonlinear optics (see \cite{KivsLut1}). Note that boundary condition \eqref{bdinfini} ensures that \eqref{GP} has a truly nonlinear dynamics, contrary to the case of null condition at infinity where the dynamics is governed by dispersion and scattering. In particular, \eqref{GP} has nontrivial localized coherent structures called ``solitons''.

At least on a formal level, the Gross-Pitaevskii equation is hamiltonian. The conserved Hamiltonian is a Ginzburg-Landau energy, namely
\begin{equation}
\label{GLE}
E(\Psi) = \frac{1}{2} \int_{\R} |\partial_\x \Psi|^2 + \frac{1}{4} \int_{\R} (1 - |\Psi|^2)^2 \equiv \int_{\R} e(\Psi).
\end{equation}
In this paper, we will only consider finite energy solutions to \eqref{GP}. Similarly, as far as it might be defined, the momentum
\begin{equation}
\label{VectP}
P(\Psi) = \frac{1}{2} \int_{\R} \langle i \partial_\x \Psi, \Psi \rangle
\end{equation}
is formally conserved. Another quantity which is formally conserved by the flow is the mass
\begin{equation}
\label{alamasse}
m(\Psi) = \frac{1}{2} \int_{\R} \Big( |\Psi|^2 - 1 \Big).
\end{equation}

Equation \eqref{GP} is integrable by means of the inverse scattering method, and it has been formally analyzed within this framework in \cite{ShabZak2}, and rigorously in \cite{GeraZha1}. The formalism of inverse scattering provides an infinite number of invariant functionals for the Gross-Pitaevskii equation and our proofs rely crucially on several of them. Concerning the Cauchy problem, it can be shown (see \cite{Zhidkov1,Gerard2}) that \eqref{GP} is locally well-posed in the spaces
$$X^k(\R) = \{ u \in L^1_{\rm loc}(\R, \C), \ {\rm s.t.} \ E(u) < + \infty, \ {\rm and} \ \partial_\x u \in H^{k - 1}(\R) \},$$
for any $k \geq 1$, and globally well-posed for $k = 1.$ In the one-dimensional case considered here, it is also globally well-posed for $k\geq 2$.

\begin{theorem}
\label{thm:existe}
Let $k \in \N^*$ and $\Psi_0 \in X^k(\R)$. Then, there exists a unique solution $\Psi(\cdot, t)$ in $\boC^0(\R, X^k(\R))$ to \eqref{GP} with initial data $\Psi_0$. If $\Psi_0$ belongs to $X^{k+2}(\R)$, then the map $t \mapsto \Psi(\cdot, t)$ belongs to $\boC^1(\R, X^k(\R))$ and $\boC^0(\R, X^{k+2}(\R))$. Moreover, the flow map $\Psi_0 \mapsto \Psi(\cdot, T)$ is continuous on $X^k(\R)$ for any fixed $T \in \R$.
\end{theorem}

Furthermore, the energy is conserved along the flow, as well as the momentum, at least under suitable assumptions (see e.g. \cite{BeGrSaS1}). On the other hand, the rigorous derivation of conservation of mass raises some difficulties.

If $\Psi$ does not vanish, one may write
$$\Psi = \sqrt{\rho} \exp i \varphi.$$
This leads to the hydrodynamic form of the equation, with $v = 2 \partial_\x \varphi,$
\begin{equation}
\label{HDGP}
\left\{ \begin{array}{ll} \partial_t \rho + \partial_\x (\rho v) = 0,\\
\rho (\partial_t v + v.\partial_\x v) + \partial_\x (\rho^2) = \rho \partial_\x \Big( \frac{ \partial_\x^2 \rho}{\rho} - \frac{|\partial_\x \rho|^2}{2 \rho^2} \Big). \end{array} \right.
\end{equation}
If one neglects the right-hand side of the second equation, which is often referred to as the quantum pressure, system \eqref{HDGP} yields the Euler equation for a compressible fluid, with pressure law $p(\rho) = \rho^2$. Since the right-hand side of \eqref{HDGP} contains third order derivatives, this approximation is only relevant in the long-wave limit. A rigorous derivation of this asymptotics was derived by Grenier in \cite{Grenier1} for different conditions at infinity.

Recall that linearizing the compressible Euler equation with pressure $p(\rho) = \rho^2$, around the constant solution $\rho = 1$ and $v = 0$, one obtains the system
\begin{equation}
\label{ondes}
\left\{ \begin{array}{ll} \partial_t \uprho + \partial_\x \v = 0,\\
\partial_t \v + 2 \partial_\x \uprho = 0, \end{array} \right.
\end{equation}
which is equivalent to the wave equation with speed $c_s$ given by
$$c_s^2 = 2.$$
This speed is referred as the sound speed for the Gross-Pitaevskii equation. In this setting, the wave equation \eqref{ondes} appears as an approximation of the Gross-Pitaevskii equation. As mentioned above, this amounts however to neglect the quantum pressure, coming from the dispersive properties of the Schr\"odinger equation, as well as to restrict ourselves to small long-wave data, so that the wave equation approximates the Euler equation. Rigorous mathematical evidence of this fact is provided in \cite{BetDaSm1}.

In order to specify the nature of the perturbation as well as of the long-wave asymptotics, we introduce a small parameter $0 < \eps < 1$ and set
$$\left\{ \begin{array}{ll} \rho(\x, t) = 1 + \frac{\eps}{\sqrt{2}} a_\eps(\eps \x,\eps t),\\
v(\x, t) = \eps v_\eps(\eps \x, \eps t), \end{array} \right.$$
so that system \eqref{HDGP} translates into
\begin{equation}
\label{eq:dynaslow}
\left\{ \begin{array}{ll} \partial_t a_\eps + \sqrt{2} \partial_\x v_\eps = - \eps \partial_\x (a_\eps v_\eps),\\
\partial_t v_\eps + \sqrt{2} \partial_\x a_\eps = \eps \bigg( - v_\eps \cdot \partial_\x v_\eps + 2 \partial_\x \Big(
\frac{\partial_\x^2 \sqrt{\sqrt{2} + \eps a_\eps}}{\sqrt{\sqrt{2} + \eps a_\eps}} \Big) \bigg). \end{array} \right.
\end{equation}

Specifying a result of \cite{BetDaSm1} in dimension one, we are led to

\begin{theorem}[\cite{BetDaSm1}]
\label{thm:un}
Let $s \geq 2$. There exists some positive constant $K(s)$ such that, given any initial datum $(a^0_\eps, v^0_\eps) \in H^{s + 1}(\R) \times H^s(\R)$ verifying
$$K(s) \eps \| (a^0_\eps, v^0_\eps) \|_{H^{s + 1}(\R) \times H^s(\R)} \leq 1,$$
there exists some real number
$$T_\eps \geq \frac{1}{K(s) \eps^2 \|(a^0_\eps, v^0_\eps) \|_{H^{s+1}(\R) \times H^s(\R)}},$$
such that system \eqref{eq:dynaslow} has a unique solution $(a_\eps, v_\eps) \in \boC^0([0, \eps T_\eps], H^{s+1}(\R) \times H^s(\R))$ satisfying
$$\| (a_\eps(\cdot, \eps t), v_\eps(\cdot, \eps t)) \|_{H^{s + 1}(\R) \times H^s(\R)} \leq K(s) \| (a^0_\eps, v^0_\eps) \|_{H^{s + 1}(\R) \times H^s(\R)}, \ {\rm and} \ \frac{1}{2} \leq \rho(\cdot, t) \leq 2,$$
for any $t \in [0, T_\eps]$. Let $(\a, \v)$ denote the solution of the free-wave equation
\begin{equation}
\label{eq:wamu0}
\left\{ \begin{array}{ll} \partial_t \a + \sqrt{2} \partial_\x \v = 0,\\
\partial_t \v + \sqrt{2} \partial_\x \a = 0, \end{array} \right.
\end{equation}
with initial datum $(a^0_\eps, v^0_\eps)$, then for any $0 \leq t \leq T_\eps$, we have
\begin{align*}
& \| (a_\eps, v_\eps)(\cdot, \eps t) - (\a, \v)(\cdot, \eps t) \|_{H^{s-2}(\R) \times H^{s-2}(\R)}\\
\leq K(s) \Big( \eps^2 t \| (a^0_\eps, & v^0_\eps) \|_{H^{s + 1}(\R) \times H^s(\R)}^2 + \eps^3 t \|(a^0_\eps, v^0_\eps)\|_{H^{s + 1}(\R) \times H^s(\R)} \Big).
\end{align*}
\end{theorem}

\begin{remark}
\label{rbs}
Notice that the bounds on $K(s)$ provided by the proof of Theorem \ref{thm:un} in \cite{BetDaSm1} blow up as $s$ tends to $+ \infty$. An interesting question is therefore to determine whether the constant $K(s)$ may be bounded independently of $s$. In particular, it would be of interest to extend the result to the limiting case $s = + \infty$.
\end{remark}

The purpose of the present paper is to consider even smaller perturbations of the constant one, and to characterize the deviation from the wave equation on larger time scales. Our initial data has the form
$$\left\{ \begin{array}{ll}
\rho(\x, 0) = 1 - \frac{\eps^2}{6} N_\eps^0(\eps \x),\\
v(\x, 0) = \frac{\eps^2}{6 \sqrt{2}} W_\eps^0(\eps \x), \end{array} \right.$$
where $N_\eps^0$ and $W_\eps^0$ are uniformly bounded in some Sobolev space $H^s(\R)$ for sufficiently large $s$. Applying Theorem \ref{thm:un} to such data, that is for $a_\eps^0 = - \frac{\eps \sqrt{2} }{6} N_\eps^0$ and $v_\eps^0 = \frac{\eps}{6 \sqrt{2}} W_\eps^0$, yields uniform bounds on a time scale $T_\eps = \boO(\eps^{- 3})$. More precisely, setting
$$n_\eps(\eps \x, \eps t) = - \frac{6}{\eps \sqrt{2}} a_\eps(\eps \x, \eps t), \ {\rm and} \ w_\eps(\eps \x, \eps t) = \frac{6 \sqrt{2}}{\eps} v_\eps(\eps \x, \eps t),$$
it follows for such initial data from Theorem \ref{thm:un} that we have

\begin{prop}
\label{prop:propre}
Assume $s \geq 2$ and $K \eps^2 \| (N_\eps^0, W_\eps^0) \|_{H^{s + 1(\R)} \times H^s(\R)} \leq 1$. Let $(\n, \w)$ denote the solution of the free wave equation
\begin{equation}
\label{eq:wamu}
\left\{ \begin{array}{ll} \partial_t \big( \sqrt{2} \n \big) - \partial_\x \w = 0,\\
\partial_t \w - 2 \partial_\x \big( \sqrt{2} \n \big) = 0, \end{array} \right.
\end{equation}
with initial datum $(N_\eps^0, W_\eps^0)$. Then, for any $0 \leq t \leq T_\eps$, we have
\begin{equation}
\begin{split}
\label{eq:wavegood}
& \|(n_\eps, w_\eps)(\cdot, \eps t) - (\n, \w)(\cdot, \eps t)\|_{H^{s-2}(\R) \times H^{s-2}(\R)}\\
\leq K \eps^3 t \Big( \|(N^0_\eps, & W^0_\eps) \|_{H^{s + 1}(\R) \times H^s(\R)} + \| (N^0_\eps, W^0_\eps) \|^2_{H^{s + 1}(\R) \times H^s(\R)} \Big),
\end{split}
\end{equation}
where
$$T_\eps = \frac{1}{K \eps^3 \| (N_\eps^0, W_\eps^0) \|_{H^{s + 1}(\R) \times H^s(\R)}}.$$
\end{prop}

In particular, if $N_\eps^0$ and $W_\eps^0$ are required to be uniformly bounded in $H^{s + 1}(\R) \times H^s(\R)$, then in view of \eqref{eq:wavegood}, the wave equation provides a good approximation on time scales of order $o(\eps^{- 3})$. This approximation ceases to be valid for times of order $\boO(\eps^{- 3})$ as the subsequent analysis will show.

The general solution to \eqref{eq:wamu} may be written as
$$(\n, \w) = (\n^+, \w^+) + (\n^-, \w^-),$$
where the functions $(\n^\pm, \w^\pm)$ are solutions to \eqref{eq:wamu} given by the d'Alembert formulae,
\begin{align*}
\big( \n^+(\x, t), \w^+(\x, t) \big) = \big( N^+(\x - \sqrt{2} t), W^+( \x - \sqrt{2} t) \big),\\
\big( \n^-(\x, t), \w^-(\x, t) \big) = \big( N^-(\x + \sqrt{2} t), W^-( \x + \sqrt{2} t) \big),
\end{align*}
where the profiles $N^\pm$ and $W^\pm$ are real-valued functions on $\R$.
Solutions may therefore be split into right and left going waves of speed
$\sqrt{2}$. Since the functions $(\n^\pm, \w^\pm)$ are solutions to \eqref{eq:wamu}, it follows that
\begin{equation}
\big( 2 N^+ + W^+ \big)_\x = 0, \ {\rm and} \ \big( 2 N^- - W^- \big)_\x = 0,
\end{equation}
so that, if the functions decay to zero at infinity, then
\begin{equation}
\label{heuriger}
2 N^\pm = \mp W^\pm = \frac{2 N_\eps^0 \mp W_\eps^0}{2}.
\end{equation}
At this stage, it is worthwhile to notice that the Gross-Pitaevskii equation, as well as the wave equation, is invariant under the symmetry $x \to - x$.

It remains to derive the appropriate approximation for time scales of order $\boO(\eps^{- 3})$. On a formal level, this was performed in \cite{KuznZak1}. We wish to give here a rigorous proof of that approximation. In view of the previous discussion, and following the approach of \cite{KuznZak1}, we introduce the slow variables
\begin{equation}
\label{martinsepp}
x = \varepsilon(\x + \sqrt{2} t), \ {\rm and} \ \tau = \frac{\varepsilon^3}{2 \sqrt{2}} t.
\end{equation}
The definition of the new variable $x$ corresponds to a reference frame travelling to the left with speed $\sqrt{2}$ in the original variables $(\x, t)$. In this frame, the wave $(\n^-, \w^-)$, originally travelling to the left, is now stationary, whereas the wave $(\n^+,\w^+)$ travelling to the right now has a speed equal to $8 \eps^{- 2}$. This change of variable is therefore particularly appropriate for the study of waves travelling to the left. This will lead us to impose some additional assumptions which will imply the smallness of $N^+$ and $W^+$. Notice that the change of frame breaks the symmetry of the original equations.

In view of \eqref{martinsepp}, we then define the rescaled functions $N_\varepsilon$ and $\Theta_\varepsilon$ as follows
\begin{equation}
\label{slow-var}
\begin{split}
N_\varepsilon(x, \tau) & = \frac{6}{\varepsilon^2} \eta(\x, t) = \frac{6}{\varepsilon^2} \eta \Big( \frac{x}{\varepsilon} - \frac{4 \tau}{\varepsilon^3}, \frac{2 \sqrt{2} \tau}{\varepsilon^3} \Big),\\
\Theta_\varepsilon(x, \tau) & = \frac{6 \sqrt{2}}{\varepsilon} \varphi(\x, t) =
\frac{6 \sqrt{2}}{\varepsilon} \varphi \Big( \frac{x}{\varepsilon} - \frac{4
\tau}{\varepsilon^3}, \frac{2 \sqrt{2} \tau}{\varepsilon^3} \Big),
\end{split}
\end{equation}
where $\Psi = \varrho \exp i \varphi$ and $\eta = 1 - \varrho^2= 1 - |\Psi|^2$.

Our main theorem is

\begin{theorem}
\label{cochon}
Let $\varepsilon > 0$ be given and assume that the initial data $\Psi_0(\cdot) = \Psi(\cdot, 0)$ belongs to $X^4(\R)$ and satisfies the assumption
\begin{equation}
\label{grinzing1}
\| N_\varepsilon^0 \|_{H^3(\R)} + \eps \| \partial^4_x N_\varepsilon^0 \|_{L^2(\R)}+ \|\partial_x \Theta_\eps^0\|_{H^3(\R)} \leq K_0.
\end{equation}
Let $\boN_\varepsilon$ and $\boM_\eps$ denote the solutions to the Korteweg-de Vries equation
\renewcommand{\theequation}{KdV}
\begin{equation}
\label{KdV}
\partial_\tau N + \partial_x^3 N + N \partial_x N = 0
\end{equation}
with initial data $N_\varepsilon^0$, respectively $\partial_x \Theta_\eps^0$. There exists positive constants $\varepsilon_0$ and $K_1$, depending possibly only on $K_0$ such that, if $\varepsilon \leq \varepsilon_0$, we have for any $\tau \in \R$,
\renewcommand{\theequation}{\arabic{equation}}
\setcounter{equation}{18}
\begin{equation}
\begin{split}
\label{eq:fortis}
\| \boN_\varepsilon & (\cdot, \tau) - N_\varepsilon(\cdot, \tau ) \|_{L^2(\R)} + \| \boM_\eps(\cdot, \tau) - \partial_x \Theta_\varepsilon(\cdot, \tau ) \|_{L^2(\R)}\\
& \leq K_1 \big( \eps + \| N_\eps^0 - \partial_x \Theta_\eps^0 \|_{H^3(\R)} \big) \exp (K_1 |\tau|).
\end{split}
\end{equation}
\end{theorem}

Theorem \ref{cochon} yields a convergence result to the \eqref{KdV} equation for appropriate initial data. Since the norms involved in \eqref{eq:fortis} are translation invariant, the \eqref{KdV} approximation can only be relevant if the waves travelling to the right are negligible. In view of our previous discussion, this is precisely the role of the term $\| N_\eps^0 -\partial_x \Theta_\eps^0 \|_{H^3(\R)}$ in the right-hand side of \eqref{eq:fortis}. Indeed, in the setting of Theorem \ref{cochon}, the right going waves $N^+$ and $W^+$ are given by
$$2 N^+ = - W^+ = N_\eps^0 - \partial_x \Theta_\eps^0.$$
If the term $\| N_\eps^0 -\partial_x \Theta_\eps^0 \|_{H^3(\R)}$ is small, then the \eqref{KdV} approximation is valid on a time interval (in the original time variable) $t \in [0, S_\eps]$ with
$$S_\eps = o \bigg( \min \bigg\{ \frac{|\log(\eps)|}{\eps^3}, \frac{|\log(\|N_\eps^0 - \partial_x \Theta_\eps^0 \|_{H^3(\R)})|}{\eps^3} \bigg\} \bigg).$$
In particular, if $\| N_\eps^0 - \partial_x \Theta_\eps^0 \|_{H^3(\R)} \leq C \eps^\alpha$, with $\alpha > 0$, then the approximation is valid on a time interval $t \in [0, S_\eps]$ with $S_\eps = o(\eps^{- 3} |\log(\eps)|)$. Moreover, if $\| N_\eps^0 - \partial_x \Theta_\eps^0 \|_{H^3(\R)}$ is of order $\boO(\eps)$, then the approximation error remains of order $\boO(\eps)$ on a time interval $t \in [0, S_\eps]$ with $S_\eps = \boO(\eps^{- 3})$.

\begin{remark}
We also show in the course of our proofs (see Proposition \ref{H3-control} below) that, under the assumptions of Theorem \ref{cochon}, the $H^3$-norms of $N_\eps$ and $\partial_x \Theta_\eps$ remain uniformly bounded in time. Since the same property holds for the solutions $\boN_\eps$ and $\boM_\eps$, it follows by interpolation that the difference of the two solutions may also be computed in terms of $H^s$-norm as
\begin{align*}
\| \boN_\varepsilon & (\cdot, \tau) - N_\varepsilon(\cdot, \tau ) \|_{H^s(\R)} + \| \boM_\eps(\cdot, \tau) - \partial_x \Theta_\varepsilon(\cdot, \tau ) \|_{H^s(\R)}\\
& \leq K \big( \eps + \| N_\eps^0 - \partial_x \Theta_\eps^0 \|_{H^3(\R)} \big)^{\alpha(s)} \exp (\alpha(s) K_1 |\tau|).
\end{align*}
for any $0 \leq s < 3$ and any $\tau \in \R$, where $\alpha(s) = 1 - \frac{s}{3}$, and where the constant $K$ depends possibly on $K_0$ and $s$.
\end{remark}

\begin{remark}
As a matter of fact, we believe that for any $s \geq 0$, the following inequality holds
 \begin{equation}
\label{fortisse3}
\| \boN_\varepsilon(\cdot, \tau) - N_\varepsilon(\cdot, \tau ) \|_{H^s(\R)} \leq K(s) \big( \eps^2 + \| N_\eps^0 - \partial_x \Theta_\eps^0 \|_{H^{s+3}(\R)} \big) \exp ( K_1 |\tau|),
\end{equation}
for any $\tau \in \R$. To prove inequality \eqref{fortisse3} along the lines of the proof of Theorem \ref{cochon} would require a more general treatment of the invariants of the Gross-Pitaevskii equation, whereas in this paper, we have only handled the lower order ones (at the cost of sometimes tedious computations). In a forthcoming paper \cite{BeGrSaS3}, we make use of a different strategy avoiding invariants but at the cost of a higher loss of derivatives. Here also, as in Remark \ref{rbs}, it would be of interest to prove a result in $H^\infty(\R)$.
\end{remark}

The functions $N_\eps$ and $\partial_x \Theta_\eps$ are rigidly constrained one to the other as shown by the following

\begin{theorem}
\label{H3-controlbis}
Let $\Psi$ be a solution to \eqref{GP} in $\boC^0(\R, H^4(\R))$ with initial data $\Psi^0$. Assume that \eqref{grinzing1} holds. Then, there exists some positive constant $K$, which does not depend on $\varepsilon$ nor $\tau$, such that
\begin{equation}
\label{dobling1ter}
\| N_\varepsilon(\cdot, \tau) \pm \partial_x \Theta_\varepsilon(\cdot, \tau) \|_{L^2(\R)} \leq \| N_\varepsilon^0 \pm \partial_x \Theta_\varepsilon^0 \|_{L^2(\R)} + K \eps ^2 \big( 1 + |\tau| \big),
\end{equation}
for any $\tau \in \R$.
\end{theorem}

The approximation errors provided by Theorem \ref{cochon} and \ref{H3-controlbis} diverge as time increases. Concerning the weaker notion of consistency, we have the following result whose peculiarity is that the bounds are independent of time.

\begin{theorem}
\label{Bobby}
Let $\Psi$ be a solution to \eqref{GP} in $\boC^0(\R, H^4(\R))$ with initial data $\Psi^0$. Assume that \eqref{grinzing1} holds. Then, there exists some positive constant $K$, which does not depend on $\varepsilon$ nor $\tau$, such that
\begin{equation}
\label{jerrard}
\| \partial_\tau U_\eps + \partial^3_x U_\eps + U_\eps \partial_x U_\eps \|_{L^2(\R)} \leq K(\eps + \|N_\eps^0 - \partial_x \Theta_\eps^0 \|_{H^3(\R)}),
\end{equation}
for any $\tau \in \R$, where $U_\eps = \frac{N_\eps + \partial_x \Theta_\eps}{2}$.
\end{theorem}

The relevance of the function $U_\eps$ will be discussed below.

A typical example where the assumptions of Theorem \ref{cochon} apply is provided by travelling wave solutions to \eqref{GP}, i.e. solutions of the form $\Psi(\x, t) = v_c(\x + c t)$, where the profile $v_c$ is a complex-valued function defined on $\R$ satisfying a simple ordinary differential equation which may be integrated explicitly. Solutions then do exist for any value of the speed $c$ in the interval $[0, \sqrt{2})$. Next, we choose the wave-length parameter to be $\varepsilon = \sqrt{2 - c^2}$, and take as initial data $\Psi_\varepsilon$ the corresponding wave $v_c$. We consider the rescaled function
$$\upnu_\varepsilon(x) = \frac{6}{\varepsilon^2} \eta_c \Big( \frac{\x}{\varepsilon} \Big),$$
where $\eta_c \equiv 1 - |v_c|^2$. The explicit integration of the travelling wave equation for $v_c$ leads to the formula
$$\upnu_\varepsilon(x) = \upnu(x) \equiv \frac{3}{\ch^2 \big( \frac{x}{2} \big)}.$$
The function $\upnu$ is the classical soliton to the Korteweg-de Vries equation, which is moved by the \eqref{KdV} flow with constant speed equal to $1$, so that
$$\boN_\varepsilon(x, \tau) = \upnu(x - \tau).$$
On the other hand, we deduce from \eqref{slow-var} that $N^0_\varepsilon = \upnu$, so that
$$N_\varepsilon(x, \tau) = \upnu \Big( x -\frac{4}{\eps^2} \Big( 1 - \sqrt{1 - \frac{\eps^2}{2}} \Big) \tau \Big).$$
Therefore, we have for any $\tau\in \R$,
$$\| \boN_\varepsilon(\cdot, \tau) - N_\varepsilon(\cdot, \tau) \|_{L^2(\R)} = \boO(\eps^2 \tau).$$
Concerning the phase $\varphi_c$ of $v_c$, we consider the scale change
$$\Theta_\varepsilon^0(x) = \frac{6 \sqrt{2}}{\varepsilon} \varphi_c \Big( \frac{x}{\varepsilon} \Big),$$
so that, in view of \cite{Graveja4},
$$\partial_x \Theta_\varepsilon^0(x) = \sqrt{1 - \frac{\varepsilon^2}{2}} \frac{\upnu(x)}{1 - \frac{\varepsilon^2}{6} \upnu(x)},$$
and hence,
$$\| N_\varepsilon^0 - \partial_x \Theta_\varepsilon^0 \|_{H^3(\R)} = \boO(\varepsilon^2).$$
This may suggest that the $\varepsilon$ error in inequality \eqref{eq:fortis} is not optimal. As a matter of fact, we believe that the optimal error term would be of order $\varepsilon^2$ (as mentioned in formula \eqref{fortisse3}). A proof of this claim would require to have higher order bounds on $N_\eps$ and $\partial_x \Theta_\eps$.

We next present some ideas in the proofs.
We infer from \eqref{GP} the equations for $N_\varepsilon$ and
$\Theta_\varepsilon$, namely
\begin{equation}
\label{slow1-0}
\partial_x N_\varepsilon - \partial_x^2 \Theta_\varepsilon + \frac{\varepsilon^2}{2} \Big( \frac{1}{2} \partial_\tau N_\varepsilon + \frac{1}{3} N_\varepsilon \partial_x^2 \Theta_\varepsilon + \frac{1}{3} \partial_x N_\varepsilon \partial_x \Theta_\varepsilon \Big) = 0,
\end{equation}
and
\begin{equation}
\label{slow2-0}
\partial_x \Theta_\varepsilon - N_\varepsilon + \frac{\varepsilon^2}{2} \Big( \frac{1}{2} \partial_\tau \Theta_\varepsilon + \frac{\partial_x^2 N_\varepsilon}{1 - \frac{\varepsilon^2}{6} N_\varepsilon} + \frac{1}{6} (\partial_x \Theta_\varepsilon)^2 \Big) + \frac{\varepsilon^4}{24} \frac{(\partial_x N_\varepsilon)^2}{(1 - \frac{\varepsilon^2}{6} N_\varepsilon)^2} = 0.
\end{equation}
The leading order in this expansion is provided by $N_\varepsilon - \partial_x \Theta_\varepsilon$ and its spatial derivative, so that an important step is to keep control on this term. In view of \eqref{slow1-0} and \eqref{slow2-0} and d'Alembert decomposition \eqref{heuriger}, we are led to introduce the new variables $U_\varepsilon$ and $V_\varepsilon$ defined by
$$U_\varepsilon = \frac{N_\varepsilon + \partial_x \Theta_\varepsilon}{2}, \ {\rm and} \ V_\varepsilon = \frac{N_\varepsilon - \partial_x \Theta_\varepsilon}{2},$$
and compute the relevant equations for $U_\varepsilon$ and $V_\varepsilon$,
\begin{equation}
\label{slow1}
\partial_\tau U_\varepsilon + \partial_x^3 U_\varepsilon + U_\varepsilon \partial_x U_\varepsilon = - \partial_x^3 V_\varepsilon + \frac{1}{3}\partial_x \big( U_\varepsilon V_\varepsilon \big) + \frac{1}{6} \partial_x \big( V_\varepsilon^2 \big) - \varepsilon^2 R_{\varepsilon},
\end{equation}
and
\begin{equation}
\label{slow2}
\partial_\tau V_\varepsilon + \frac{8}{\varepsilon^2} \partial_x V_\varepsilon = \partial_x^3 U_\varepsilon + \partial_x^3 V_\eps + \frac{1}{2} \partial_x (V_\varepsilon^2) - \frac{1}{6} \partial_x (U_\varepsilon)^2 -\frac{1}{3} \partial_x (U_\varepsilon V_\varepsilon) + \varepsilon^2 R_{\varepsilon},
\end{equation}
where the remainder term $R_{\varepsilon}$ is given by the formula
\begin{equation}
\label{grouin}
R_{\varepsilon} = \frac{N_\varepsilon\partial_x^3 N_\varepsilon}{6 (1 - \frac{\varepsilon^2}{6} N_\varepsilon)} + \frac{(\partial_x N_\varepsilon) (\partial_x^2 N_\varepsilon)}{3 (1 - \frac{\varepsilon^2}{6} N_\varepsilon)^2} + \frac{\varepsilon^2}{36} \frac{(\partial_x N_\varepsilon)^3}{(1 - \frac{\varepsilon^2}{6} N_\varepsilon)^3}.
\end{equation}
The left-hand side of equation \eqref{slow1} corresponds to the \eqref{KdV} operator applied to $U_\varepsilon$: a major step in the proof is therefore to establish that the right-hand side is small in suitable norms. This amounts in particular, as already mentioned, to show that $V_\varepsilon$, which is assumed to be small at time $\tau = 0$ remains small, and that $U_\varepsilon$, which is assumed to be bounded at time $\tau = 0$, remains bounded in appropriate Sobolev norm. To establish these estimates, we rely among other things on several conservation laws which are provided by the integrability of the one-dimensional \eqref{GP} equation. To illustrate the argument, we next present it for the $L^2$-norm, where we only need to invoke the conservation of energy and momentum.

In the rescaled setting, the Ginzburg-Landau energy may be written as
\begin{equation}
\label{sirbu1}
E(\Psi) = \frac{\varepsilon^3}{144} \Bigg( \int_\R \Big( (\partial_x \Theta_\varepsilon)^2 + N_\varepsilon^2 \Big) + \frac{\varepsilon^2}{2} \int_\R \bigg( \frac{(\partial_x N_\varepsilon)^2}{1 - \frac{\varepsilon^2}{6} N_\varepsilon} -\frac{1}{3} N_\varepsilon(\partial_x \Theta_\varepsilon)^2 \bigg) \Bigg)
\equiv \frac{\varepsilon^3}{18} \boE_1(N_\varepsilon, \Theta_\varepsilon),
\end{equation}
so that assumption \eqref{grinzing1} implies that
\begin{equation}
\label{pauli1}
\boE_1(N_\varepsilon^0, \Theta_\varepsilon^0) \leq K_0.
\end{equation}
On the other hand, when the energy $E(\Psi)$ is sufficiently small, which is the case at the limit $\eps \to 0$, we may assume that
$$\frac{1}{2} \leq |\Psi| \leq 2,$$
which may be translated as
\begin{equation}
\label{borninf}
\frac{1}{4} \leq 1 - \frac{\varepsilon^2}{6} N_\varepsilon \leq 4,
\end{equation}
so that the rescaled energy $\boE_1$ satisfies
\begin{equation}
\label{L2-control}
\int_\R \Big( (\partial_x \Theta_\varepsilon)^2 + N_\varepsilon^2 \Big) \leq K \boE_1(N_\varepsilon, \Theta_\varepsilon),
\end{equation}
where $K$ is some universal constant. Similarly, the momentum may be written as
\begin{equation}
\label{sirbu2}
P(\Psi) = \frac{1}{2} \int_\R \eta \partial_\x \varphi = \frac{\varepsilon^3}{72 \sqrt{2}} \int_\R N_\varepsilon \partial_x \Theta_\varepsilon \equiv \frac{\varepsilon^3}{18} \boP_1(N_\varepsilon, \Theta_\varepsilon).
\end{equation}
Next, we compute
$$\boE_1(N_\varepsilon, \Theta_\varepsilon) - \sqrt{2} \boP_1(N_\varepsilon, \Theta_\varepsilon) = \frac{1}{8} \int_\R (N_\varepsilon - \partial_x \Theta_\varepsilon)^2 + \frac{\varepsilon^2}{8} \int_\R \bigg( \frac{\partial_x N_\varepsilon^2}{1 - \frac{\varepsilon^2}{6} N_\varepsilon} - \frac{1}{3} N_\varepsilon (\partial_x \Theta_\varepsilon)^2 \bigg),$$
so that
\begin{equation}
\label{wachovien}
\Big| \boE_1(N_\varepsilon^0, \Theta_\varepsilon^0) - \sqrt{2} \boP_1(N_\varepsilon^0, \Theta_\varepsilon^0) \Big| \leq K_0.
\end{equation}
Moreover, by the Sobolev embedding theorem and the inequality $2 a b \leq a^2 + b^2$,
\begin{equation}
\label{eq:wachovia}
\boE_1(N_\varepsilon, \Theta_\varepsilon) - \sqrt{2} \boP_1(N_\varepsilon,
\Theta_\varepsilon) \geq K^{-1} \Big( \int_\R V_\varepsilon^2 + \varepsilon^2 \int_\R (\partial_x N_\varepsilon)^2\Big) - K \varepsilon^2 \bigg( \int_\R (\partial_x \Theta_\varepsilon)^2 \bigg)^2,
\end{equation}
where $K$ refers to some universal constant. By conservation, we then have
\begin{equation}
\label{conserve}
\frac{d}{d\tau} \big( \boE_1(N_\varepsilon, \Theta_\varepsilon) \big) = 0, \ {\rm and } \ \frac{d}{d\tau} \big( \boP_1(N_\varepsilon, \Theta_\varepsilon) \big) = 0.
\end{equation}
Invoking \eqref{pauli1} and \eqref{L2-control}, we are led to
$$\| N_\eps(\cdot, \tau)\|_{L^2(\R)}^2 + \|\partial_x \Theta_\eps(\cdot, \tau) \|_{L^2(\R)}^2 \leq K_0,$$
for any $\tau \in \R$. In turn, using \eqref{wachovien}, \eqref{eq:wachovia} and \eqref{conserve} yields
$$\|V_\eps(\cdot, \tau)\|_{L^2(\R)}^2 \leq K \big( \|V_\eps(0)\|_{L^2(\R)}^2 + \eps^2 \big).$$

It turns out that the other conservation laws for the Gross-Pitaevskii equation involve quantities which behave as higher order energies and others which behave as higher order momenta. We denote $\boE_k(N_\eps, \Theta_\eps)$ and $\boP_k(N_\eps, \Theta_\eps)$, respectively these quantities (precise expressions are provided in Section \ref{Invariants}). Using these invariants, we may perform a similar argument to control higher Sobolev norms. This gives

\begin{prop}
\label{H3-control}
Let $\Psi$ be a solution to \eqref{GP} in $\boC^0(\R, H^4(\R))$ with initial data $\Psi^0$. Assume that \eqref{grinzing1} holds. Then, there exists some positive constant $K$, which does not depend on $\varepsilon$ nor $\tau$, such that
\begin{equation}
\label{grinzing1bis}
\| N_\varepsilon(\cdot, \tau) \|_{H^3(\R)} + \eps \| \partial_x^4 N_\varepsilon(\cdot, \tau) \|_{L^2(\R)} + \|\partial_x \Theta_\eps(\cdot, \tau) \|_{H^3(\R)} \leq K,
\end{equation}
and
\begin{equation}
\label{dobling1bis}
\| N_\varepsilon(\cdot, \tau) \pm \partial_x \Theta_\varepsilon(\cdot, \tau) \|_{H^3(\R)} \leq K \big( \| N_\varepsilon^0 \pm \partial_x \Theta_\varepsilon^0 \|_{H^3(\R)} + \eps \big),
\end{equation}
for any $\tau \in \R$.
\end{prop}

The proof of Theorem \ref{Bobby} follows directly from Proposition \ref{H3-control}.
Using a standard energy method applied to the system \eqref{slow1-0} and \eqref{slow2-0}
and taking advantage of the fact that the left-hand side of equation \eqref{slow2-0} is a transport operator with speed $\frac{8}{\eps^2}$, we obtain Theorem \ref{H3-controlbis}.
Finally, the proof of Theorem \ref{cochon} follows again from an energy method applied to the difference $W_\eps = N_\eps - \boN_\eps$ (and the equivalent for $\partial_x \Theta_\eps$).

\begin{remark}
It is worthwhile to stress that in the course of proving Proposition \ref{H3-control}, we have been led to prove a number of facts which, we believe, are of independent interest, and represent actually the bulk contribution of our paper. First, we have given expressions of the invariant quantities and proved that they are well-defined on the spaces $X^k(\R)$: their expressions are not a straightforward consequence of the inductive formulae for the conservation laws provided by the inverse scattering method of \cite{ShabZak2}. Indeed, various renormalizations have to be applied to give a sound mathematical meaning to the expressions. Moreover, we have rigorously established that these quantities are conserved by the \eqref{GP} flow in the appropriate functional spaces.

In a related direction, we have highlighted a strong and somewhat striking relationship between the \eqref{GP} invariants and the \eqref{KdV} invariants. More precisely, we have shown that, for any $1 \leq k \leq 4$ and for any functions in the appropriate spaces,
$$\boE_k(N, \partial_x \Theta) - \sqrt{2} \boP_k(N, \partial_x \Theta) = E_k^{KdV} \Big( \frac{N - \partial_x \Theta}{2} \Big) + \boO(\eps^2),$$
where $E_k^{KdV}$ refers to the \eqref{KdV} invariants (for more precise statements, see Proposition \ref{Controluv}). In particular, the \eqref{GP} invariants $\boE_k$ and $\boP_k$, as well as the \eqref{KdV} invariants, provide control on the $H^k$-norms.
\end{remark}

\begin{remark}
It would be of interest to investigate further the relationships between \eqref{GP} and \eqref{KdV}, in particular at the level of the spectral problems associated to the corresponding inverse scattering methods. Indeed, recall that $\eqref{KdV}$ can be resolved using scattering and inverse scattering methods for the linear Schr\"odinger equation
$$L_\boN(\Phi) = -\partial_x^2 \Phi + \boN \Phi,$$
whereas $\eqref{GP}$ is known to be tractable using the scattering and inverse scattering methods for the Dirac operator
$$D_\Psi(\Phi_1, \Phi_2) = i \begin{pmatrix} 1 + \sqrt{3} & 0\\0 & 1 - \sqrt{3} \end{pmatrix} \begin{pmatrix} \partial_x \Phi_1 \\ \partial_x \Phi_2 \end{pmatrix} + \begin{pmatrix} 0 & \Psi^* \\ \Psi & 0 \end{pmatrix} \begin{pmatrix} \Phi_1 \\ \Phi_2 \end{pmatrix}.$$
Besides, it is known that the Schr\"odinger equation is a nonrelativistic limit of the Dirac equation. Kutznetsov and Zakharov \cite{KuznZak1} suggest that this correspondence can be carried out in the asymptotic limit considered here. Notice however that rigorous scattering and inverse scattering methods require decay and regularity assumptions on the data (see e.g. \cite{GeraZha1} where the datum is required to decay at least as $|x|^{- 4}$, as well as its first three derivatives). 
\end{remark}

Let us emphasize again that our paper focuses on the left going waves. Our proof requires to impose conditions on the initial data to ensure that the right going wave is small. An interesting problem is to remove this assumption, i.e. to consider simultaneously both left and right going waves, and to study their interaction. We hope to handle this problem in a forthcoming paper, as well as the already mentioned optimal bounds.

The paper is organized as follows. The next section is devoted to properties of the Cauchy problem. In Section \ref{Invariants}, we compute the invariants of the \eqref{GP} flow needed for our proofs, and show that they are conserved. In Section \ref{Rescaledinv}, we recast these invariants in the asymptotics considered here, and show the convergence to the \eqref{KdV} invariants. In Section \ref{Notime}, we give the proofs to Proposition \ref{H3-control} and Theorem \ref{Bobby}. Finally, in Section \ref{Expansion}, we present the energy methods which yield the proofs to Theorems \ref{cochon} and \ref{H3-controlbis}.

While completing this work, we learned that D. Chiron and F. Rousset \cite{ChirRou2} were obtaining at the same time several results which are related to our analysis of the \eqref{KdV} limit, and also treated the higher dimensional case.

\begin{merci}
The authors are grateful to the referee for his forward looking remarks which helped to improve the manuscript.\\
A large part of this work was completed while the four authors were visiting the Wolfgang Pauli Institute in Vienna. We wish to thank warmly this institution, as well as Prof. Norbert Mauser for the hospitality and support. We are also thankful to Dr. Martin Sepp for fruitful digressions.\\
F.B., P.G. and D.S. are partially sponsored by project JC05-51279 of the Agence Nationale de la Recherche. J.-C. S. acknowledges support from project ANR-07-BLAN-0250 of the Agence Nationale de la Recherche.
\end{merci}

\numberwithin{cor}{section}
\numberwithin{equation}{section}
\numberwithin{lemma}{section}
\numberwithin{prop}{section}
\numberwithin{remark}{section}
\numberwithin{theorem}{section}
\section{Global well-posedness for the Gross-Pitaevskii equation}
\label{Gwp}

The purpose of this section is to present the proof of Theorem \ref{thm:existe}. It is presumably well-known to the experts, but we did not find it stated in the literature, and therefore we provide a proof here for the sake of completeness.

Notice that Gallo \cite{Gallo3} already established the local well-posedness of \eqref{GP} in the spaces $X^k(\R)$ for any $k \geq 1$ (see also \cite{Zhidkov1,Gerard2}). More precisely, we have

\begin{theorem}[\cite{Gallo3,Gerard2}]
\label{Augalop}
Let $k \geq 2$. Given any function $\Psi_0 \in X^k(\R^N)$, consider the unique solution $\Psi(\cdot, t)$ to \eqref{GP} in $\boC^0(\R, X^1(\R))$ with initial data $\Psi_0$. Then, there exist $(T_-, T_+) \in (0, + \infty]^2$ such that the map $t \mapsto \Psi(\cdot, t)$ belongs to $\boC^0((- T_-, T_+), X^k(\R))$. Moreover, either $T_+$ is equal to $+ \infty$, respectively $T_- = + \infty$, or
\begin{equation}
\label{impossible}
\| \partial_\x \Psi(\cdot, t) \|_{H^{k-1}(\R)} \to + \infty, \ {\rm as} \ t \to T_+ \ ({\rm resp.} \ t \to - T_-).
\end{equation}
If $\Psi_0$ belongs to $X^{k+2}(\R)$, then the map $t \mapsto \Psi(\cdot, t)$ belongs to $\boC^1((- T_-, T_+), X^k(\R))$ and $\boC^0((- T_-, T_+), X^{k+2}(\R))$. Moreover, the flow map $\Psi_0 \mapsto \Psi(\cdot, T)$ is continuous on $X^k(\R)$ for any fixed $- T_- < T < T_+$.
\end{theorem}

In view of Theorem \ref{Augalop}, the proof of Theorem \ref{thm:existe} reduces to establish that the $H^{k - 1}$-norm of the function $\partial_\x \Psi$ cannot blow-up in finite time. In \cite{Gallo3,Gerard1}, it is proved that the linear Schr\"odinger propagator $S(t)$ maps $X^k(\R)$ into $X^k(\R)$, so that we may invoke the Duhamel formula
$$\Psi(\cdot, t) = S(t) \Psi_0 - \int_0^t S(t-s) \Psi(\cdot, s) \big( 1 - |\Psi(\cdot, s)|^2 \big) ds,$$
to estimate the $H^{k - 1}$-norm of the function $\partial_\x \Psi$ by
\begin{equation}
\label{grogne}
\| \partial_\x \Psi(\cdot, t) \|_{H^{k-1}(\R)} \leq \| \partial_\x \Psi_0 \|_{H^{k-1}(\R)} + \bigg| \int_0^t \| \partial_\x \big( \Psi(\cdot, s) \big( 1 - |\Psi(\cdot, s)|^2 \big) \big) \|_{H^{k-1}(\R)} ds \bigg|.
\end{equation}
To estimate the second term on the left-hand side, we invoke the following tame estimates.

\begin{lemma}
\label{Tamise}
Let $k \geq 1$ and $(\psi_1,\psi_2) \in X^k(\R)^2$. Given any $1 \leq j \leq k$, there exists some constant $K(j, k)$, depending only on $j$ and $k$, such that
\begin{equation}
\label{clyde}
\big\| \partial_\x^j \big( \psi_1 \psi_2 \big) \big\|_{L^2(\R)} \leq K(j, k) \Big( \| \psi_1 \|_{L^\infty(\R)} \| \partial_\x^k \psi_2 \|_{L^2(\R)} + \| \psi_2 \|_{L^\infty(\R)} \| \partial_\x^k \psi_1 \|_{L^2(\R)} \Big).
\end{equation}
\end{lemma}

We postpone the proof of Lemma \ref{Tamise} and first complete the proof of Theorem \ref{thm:existe}.

\begin{proof}[Proof of Theorem \ref{thm:existe}]
In view of \eqref{clyde}, inequality \eqref{grogne} yields
\begin{equation}
\label{wall-e}
\| \partial_\x \Psi(\cdot, t) \|_{H^{k-1}(\R)} \leq \| \partial_\x \Psi_0 \|_{H^{k-1}(\R)} + K(k) \bigg| \int_0^t \big( 1 + \| \Psi(\cdot, s) \|_{L^\infty(\R)}^2 \big) \| \partial_\x \Psi(\cdot, s) \|_{H^{k-1}(\R)} ds \bigg|,
\end{equation}
where $K(k)$ is some constant depending only on $k$. Notice that $\Psi_0$ belongs to $X^1(\R)$, so that in view of the conservation of energy proved in \cite{Gerard1} (see Theorem \ref{ConsE} below), we have
$$E(\Psi(\cdot, t)) = E(\Psi_0).$$
Next, given any function $\psi \in X^1(\R)$, there exists some universal positive constant $K$ such that
\begin{equation}
\label{holtz}
\| \psi \|_{L^\infty(\R)} \leq K \big( 1 + E(\psi) \big)^\frac{1}{2}.
\end{equation}
In particular, it follows from \eqref{holtz} that
$\| \Psi(s) \|_{L^\infty(\R)} \leq K \big( 1 + E(\Psi_0) \big)^\frac{1}{2},$
so that by \eqref{wall-e}, we are led to
$$\| \partial_\x \Psi(\cdot, t) \|_{H^{k-1}(\R)} \leq K(k, \Psi_0) \bigg( 1 + \bigg| \int_0^t \| \partial_\x \Psi(\cdot, s) \|_{H^{k-1}(\R)} ds \bigg| \bigg),$$
where $K(k, \Psi_0)$ is some constant only depending on $k$, $E(\Psi_0)$ and $\| \partial_\x \Psi_0 \|_{H^{k-1}(\R)}$. Therefore, we have by integration
$$\| \partial_\x \Psi(\cdot, t) \|_{H^{k-1}(\R)} \leq K(k, \Psi_0) \exp \big( K(k, \Psi_0) |t| \big), $$
and it follows, going back to \eqref{impossible}, that
$$T_- = T_+ = + \infty,$$
which completes the proof.
\end{proof}

We now provide the proof of Lemma \ref{Tamise}.

\begin{proof}[Proof of Lemma \ref{Tamise}]
We introduce some cut-off function $\chi \in \boC^\infty(\R, [0, 1])$ such that
\begin{equation}
\label{ness}
\chi = 1 \ {\rm on} \ (-1, 1), \ {\rm and} \ \chi = 0 \ {\rm on} \ \R \setminus (-2, 2),
\end{equation}
and set
\begin{equation}
\label{maree}
\chi_R(\x) = \chi \Big( \frac{\x}{R} \Big), \ \forall \x \in \R,
\end{equation}
for any $R > 1$. Using standard tame estimates, we have
\begin{equation}
\begin{split}
\label{coe}
& \big\| \partial_\x^j \big( \chi_R \psi_1 \chi_R \psi_2 \big) \big\|_{L^2(\R)}\\
\leq K(j, k) \Big( \| & \chi_R \psi_1 \|_{L^\infty(\R)} \| \partial_\x^k (\chi_R \psi_2) \|_{L^2(\R)} + \| \chi_R \psi_2 \|_{L^\infty(\R)} \| \partial_\x^k (\chi_R \psi_1) \|_{L^2(\R)} \Big)\\
\leq K(j, k) \Big( \| & \psi_1 \|_{L^\infty(\R)} \| \partial_\x^k (\chi_R \psi_2) \|_{L^2(\R)} + \| \psi_2 \|_{L^\infty(\R)} \| \partial_\x^k (\chi_R \psi_1) \|_{L^2(\R)} \Big).
\end{split}
\end{equation}
We now claim that
\begin{equation}
\label{affric}
\| \partial_\x^j (\chi_R \psi) \|_{L^2(\R)} \to \| \partial_\x^j \psi \|_{L^2(\R)}, \ {\rm as} \ R \to + \infty.
\end{equation}
for any function $\psi \in X^k(\R)$ and any $1 \leq j \leq k$. As a matter of fact, by the Leibniz formula, we have
\begin{equation}
\label{leibniz}
\partial_\x^j (\chi_R \psi) = \sum_{m = 1}^j C_j^m \partial_\x^m \chi_R \partial_\x^{j - m} \psi.
\end{equation}
We next deduce from the dominated convergence theorem that
$$\chi_R \partial_\x^j \psi \to \partial_\x^j \psi \ {\rm in} \ L^2(\R), \ {\rm as} \ R \to + \infty,$$
whereas, when $m \geq 1$, we similarly have using \eqref{ness} and \eqref{maree},
\begin{align*}
\int_\R \big| \partial_\x^m \chi_R \partial_\x^{j - m} \psi \big|^2 = & \frac{1}{R^{2 m-1}} \bigg( \int_1^2 \big| \partial_\x^m \chi(x) \partial_\x^{j - m} \psi(R x) \big|^2 dx + \int_{- 2}^{- 1} \big| \partial_\x^j \chi(x) \partial_\x^{j - m} \psi(R x) \big|^2 dx \bigg)\\
\leq & \frac{K}{R^{2 m - 1}} \| \partial_\x^{j - m} \psi \|_{L^\infty(\R)}^2 \to 0, \ {\rm as} \ R \to + \infty.
\end{align*}
Hence, in view of \eqref{leibniz}, we are led to
$$\partial_\x^j (\chi_R \psi) \to \partial_\x^j \psi \ {\rm in} \ L^2(\R), \ {\rm as} \ R \to + \infty,$$
which ends the proof of claim \eqref{affric}. Combining \eqref{coe} with \eqref{affric}, and noticing that \eqref{affric} remains valid replacing $\chi_R$ by $\chi_R^2$, we obtain \eqref{clyde} at the limit $R \to + \infty$. This concludes the proof of Lemma \ref{Tamise}.
\end{proof}

\section{Invariants of the Gross-Pitaevskii equations}
\label{Invariants}

\subsection{Formal derivation of the invariants}
\label{Formol}

In \cite{ShabZak2}, Shabat and Zakharov established that the one-dimensional Gross-Pitaevskii equation is integrable, and admits an infinite number of conservation laws $f_n(\Psi)$, leading to an infinite family of invariants $I_n(\Psi)$. Set 
\begin{equation}
\label{f1}
f_1(\Psi) = - \frac{1}{2} |\Psi|^2.
\end{equation}
and let
\begin{equation}
\label{recurfn}
f_{n+1}(\Psi) = \overline{\Psi} \partial_\x \Big( \frac{f_n(\Psi)}{\overline{\Psi}} \Big) + \sum_{j=1}^{n-1} f_j(\Psi) f_{n-j}(\Psi).
\end{equation}
Using the inverse scattering method, it is shown formally in \cite{ShabZak2} that the functions $f_n(\Psi)$ are conservation laws for \eqref{GP}, so that the related integral quantities $I_n(\Psi)$ defined by
\begin{equation}
\label{densities}
I_n(\Psi) = \int_{\R} \big( f_n(\Psi)(\x) - f_n(\Psi)(\infty) \big) d\x,
\end{equation}
are invariants for \eqref{GP}. Here, the notation $f_n(\Psi)(\infty)$ stands for the limit at infinity of the map $f_n(\Psi)$ assuming that
$$\Psi(\x) \to 1, \ {\rm as} \ |\x| \to + \infty, \ {\rm and} \ \partial_\x^k \Psi(\x) \to 0, \ {\rm as} \ |\x| \to + \infty,$$
for any $k \in \N^*$. The first five conservation laws are computed in \cite{ShabZak2}, namely \eqref{f1} and
\begin{align}
\label{f2}
f_2(\Psi) & = - \frac{1}{2} \overline{\Psi} \partial_\x \Psi,\\
\label{f3}
f_3(\Psi) & = - \frac{1}{2} \overline{\Psi} \partial_\x^2 \Psi + \frac{1}{4} |\Psi|^4,\\
\label{f4}
f_4(\Psi) & = - \frac{1}{2} \overline{\Psi} \partial_\x^3 \Psi + |\Psi|^2 \overline{\Psi} \partial_\x \Psi + \frac{1}{4} |\Psi|^2 \Psi \partial_\x \overline{\Psi},\\
\label{f5}
f_5(\Psi) & = - \frac{1}{2} \overline{\Psi} \partial_\x^4 \Psi + \frac{3}{2} |\Psi|^2 \overline{\Psi} \partial_\x^2 \Psi + \frac{1}{4} |\Psi|^2 \Psi \partial_\x^2 \overline{\Psi} + \frac{3}{2} |\Psi|^2 |\partial_\x \Psi|^2 + \frac{5}{4} (\overline{\Psi})^2 (\partial_\x \Psi)^2 - \frac{1}{4} |\Psi|^6.
\end{align}

The purpose of this section is to give a rigorous meaning to these quantities, to prove that they are conserved, and to extend the explicit list of invariants. As a matter of fact, these invariants enter directly in our analysis of the transonic limit.

The first step is to compute the additional conservation laws using formula \eqref{recurfn}. Notice first that formula \eqref{recurfn} is singular at the points where $\Psi$ vanishes. A first task is therefore to show that \eqref{recurfn} can be used to define the functionals $f_n(\Psi)$ even in the case the function $\psi$ vanishes somewhere. To remove the singularity in \eqref{recurfn}, we check by induction that the function $f_n(\Psi)$ may be written as
\begin{equation}
\label{formfn}
f_n(\Psi) = \overline{\Psi} \times \boF_n(\Psi),
\end{equation}
where the map $\boF_n$ is inductively defined by
\begin{equation}
\label{boF1}
\boF_1(\Psi) = - \frac{\Psi}{2},
\end{equation}
and
\begin{equation}
\label{recurFn}
\boF_{n+1}(\Psi) = \partial_\x \boF_n(\Psi) + \overline{\Psi} \sum_{j=1}^{n-1} \boF_j(\Psi) \boF_{n-j}(\Psi).
\end{equation}
In particular, the map $\boF_n(\Psi)$ is a polynomial functional of the functions $\Psi$, $\overline{\Psi}$, $\cdots$, $\partial_\x^{n-2} \Psi$, $\partial_\x^{n-2} \overline{\Psi}$ and $\partial_\x^{n-1} \Psi$, which is defined without additional assumptions on $\Psi$. This leads to explicit expressions of $f_6(\Psi)$, $f_7(\Psi)$, $f_8(\Psi)$ and $f_9(\Psi)$, which are given by
\begin{align*}
f_6(\Psi) = & - \frac{1}{2} \overline{\Psi} \partial_\x^5 \Psi + 2 |\Psi|^2 \overline{\Psi} \partial_\x^3 \Psi + \frac{1}{4} |\Psi|^2 \Psi \partial_\x^3 \overline{\Psi} + 2 |\Psi|^2 \partial_\x \Psi \partial_\x^2 \overline{\Psi} + 3 |\Psi|^2 \partial_\x \overline{\Psi} \partial_\x^2 \Psi\\
\notag & + \frac{9}{2} (\overline{\Psi})^2 \partial_\x \Psi \partial_\x^2 \Psi + \frac{11}{4} |\partial_\x \Psi|^2 \overline{\Psi} \partial_\x \Psi - \frac{3}{4} |\Psi|^4 \Psi \partial_\x \overline{\Psi} - 2 |\Psi|^4 \overline{\Psi} \partial_\x \Psi,\\
f_7(\Psi) = & - \frac{1}{2} \overline{\Psi} \partial_\x^6 \Psi + \frac{1}{4} |\Psi|^2 \Psi \partial_\x^4 \overline{\Psi} + \frac{5}{2} |\Psi|^2 \overline{\Psi} \partial_\x^4 \Psi + \frac{5}{2} |\Psi|^2 \partial_\x \Psi \partial_\x^3 \overline{\Psi} + 5 |\Psi|^2 \partial_\x \overline{\Psi} \partial_\x^3 \Psi\\
\notag & + 7 (\overline{\Psi})^2 \partial_\x \Psi \partial_\x^3 \Psi + 5 |\Psi|^2 |\partial_\x^2 \Psi|^2 + \frac{19}{4} (\partial_\x \Psi)^2 \overline{\Psi} \partial_\x^2 \overline{\Psi}+ \frac{19}{4} (\overline{\Psi})^2 (\partial_\x^2 \Psi)^2 + 13 |\partial_\x \Psi|^2 \overline{\Psi} \partial_\x^2 \Psi\\
\notag & - |\Psi|^4 \Psi \partial_\x^2 \overline{\Psi} - \frac{15}{4} |\Psi|^4 \overline{\Psi} \partial_\x^2 \Psi - \frac{3}{4} |\Psi|^2 (\Psi)^2 (\partial_\x \overline{\Psi})^2 - 8 |\Psi|^4 |\partial_\x \Psi|^2 - \frac{25}{4} |\Psi|^2 (\overline{\Psi})^2 (\partial_\x \Psi)^2\\
\notag & + \frac{5}{16} |\Psi|^8,\\
f_8(\Psi) = & - \frac{1}{2} \overline{\Psi} \partial_\x^7 \Psi + \frac{1}{4} |\Psi|^2 \Psi \partial_\x^5 \overline{\Psi} + 3 |\Psi|^2 \overline{\Psi} \partial_\x^5 \Psi + 3 |\Psi|^2 \partial_\x \Psi \partial_\x^4 \overline{\Psi} + \frac{15}{2} |\Psi|^2 \partial_\x \overline{\Psi} \partial_\x^4 \Psi\\
\notag & + 10 (\overline{\Psi})^2 \partial_\x \Psi \partial_\x^4 \Psi + \frac{15}{2} |\Psi|^2 \partial_\x^2 \Psi \partial_\x^3 \overline{\Psi} + \frac{29}{4} (\partial_\x \Psi)^2 \overline{\Psi} \partial_\x^3 \overline{\Psi} + 10 |\Psi|^2 \partial_\x^2 \overline{\Psi} \partial_\x^3 \Psi\\
\notag & + 17 (\overline{\Psi})^2 \partial_\x^2 \Psi \partial_\x^3 \Psi + 25 |\partial_\x \Psi|^2 \overline{\Psi} \partial_\x^3 \Psi + \frac{55}{2} |\partial_\x^2 \Psi|^2 \overline{\Psi} \partial_\x \Psi + \frac{71}{4} (\partial_\x^2 \Psi)^2 \overline{\Psi} \partial_\x \overline{\Psi} - \frac{5}{4} |\Psi|^4 \Psi \partial_\x^3 \overline{\Psi}\\
\notag & - 6 |\Psi|^4 \overline{\Psi} \partial_\x^3 \Psi - \frac{5}{2} |\Psi|^2 (\Psi)^2 \partial_\x \overline{\Psi} \partial_\x^2 \overline{\Psi} - \frac{53}{4} |\Psi|^4 \partial_\x \Psi \partial_\x^2 \overline{\Psi} - \frac{75}{4} |\Psi|^4 \partial_\x \overline{\Psi} \partial_\x^2 \Psi\\
\notag & - 27 |\Psi|^2 (\overline{\Psi})^2 \partial_\x \Psi \partial_\x^2 \Psi - \frac{41}{4} |\Psi|^2 |\partial_\x \Psi|^2 \Psi \partial_\x \overline{\Psi} - \frac{131}{4} |\Psi|^2 |\partial_\x \Psi|^2 \overline{\Psi} \partial_\x \Psi - \frac{15}{2} (\overline{\Psi})^3 (\partial_\x \Psi)^3\\
\notag & + \frac{29}{16} |\Psi|^6 \Psi \partial_\x \overline{\Psi} + 4 |\Psi|^6 \overline{\Psi} \partial_\x \Psi,
\end{align*}
and
\begin{align*}
f_9(\Psi) = & - \frac{1}{2} \overline{\Psi} \partial_\x^8 \Psi + \frac{1}{4} |\Psi|^2 \Psi \partial_\x^6 \overline{\Psi} + \frac{7}{2} |\Psi|^2 \overline{\Psi} \partial_\x^6 \Psi + \frac{7}{2} |\Psi|^2 \partial_\x \Psi \partial_\x^5 \overline{\Psi} + \frac{21}{2} |\Psi|^2 \partial_\x \overline{\Psi} \partial_\x^5 \Psi\\
& + \frac{27}{2} (\overline{\Psi})^2 \partial_\x \Psi \partial_\x^5 \Psi + \frac{21}{2} |\Psi|^2 \partial_\x^2 \Psi \partial_\x^4 \overline{\Psi} + \frac{41}{4} (\partial_\x \Psi)^2 \overline{\Psi} \partial_\x^4 \overline{\Psi} + \frac{35}{2} |\Psi|^2 \partial_\x^2 \overline{\Psi} \partial_\x^4 \Psi\\
& + \frac{55}{2} (\overline{\Psi})^2 \partial_\x^2 \Psi \partial_\x^4 \Psi + \frac{85}{2} |\partial_\x \Psi|^2 \overline{\Psi} \partial_\x^4 \Psi + \frac{35}{2} |\Psi|^2 |\partial_\x^3 \Psi|^2 + \frac{99}{2} \overline{\Psi} \partial_\x \Psi \partial_\x^2 \Psi \partial_\x^3 \overline{\Psi}\\
& + \frac{69}{4} (\overline{\Psi})^2 (\partial_\x^3 \Psi)^2 + \frac{125}{2} \overline{\Psi} \partial_\x \Psi \partial_\x^2 \overline{\Psi} \partial_\x^3 \Psi + \frac{155}{2} \overline{\Psi} \partial_\x \overline{\Psi} \partial_\x^2 \Psi \partial_\x^3 \Psi + \frac{181}{4} |\partial_\x^2 \Psi|^2 \overline{\Psi} \partial_\x^2 \Psi\\
& - \frac{3}{2} |\Psi|^4 \Psi \partial_\x^4 \overline{\Psi} - \frac{35}{4} |\Psi|^4 \overline{\Psi} \partial_\x^4 \Psi - \frac{15}{4} |\Psi|^2 (\Psi)^2 \partial_\x \overline{\Psi} \partial_\x^3 \overline{\Psi} - \frac{79}{4} |\Psi|^4 \partial_\x \Psi \partial_\x^3 \overline{\Psi}\\
& - 36 |\Psi|^4 \partial_\x \overline{\Psi} \partial_\x^3 \Psi - 49 |\Psi|^2 (\overline{\Psi})^2 \partial_\x \Psi \partial_\x^3 \Psi - \frac{5}{2} |\Psi|^2 (\Psi)^2 (\partial_\x^2 \overline{\Psi})^2 - \frac{149}{4} |\Psi|^4 |\partial_\x^2 \Psi|^2\\
& - \frac{165}{4} |\Psi|^2 |\partial_\x \Psi|^2 \Psi \partial_\x^2 \overline{\Psi} - 66 |\Psi|^2 \overline{\Psi} (\partial_\x \Psi)^2 \partial_\x^2 \overline{\Psi} - \frac{133}{4} |\Psi|^2 (\overline{\Psi})^2 (\partial_\x^2 \Psi)^2\\
& - 29 |\Psi|^2 \Psi (\partial_\x \overline{\Psi})^2 \partial_\x^2 \Psi - \frac{349}{2} |\Psi|^2 |\partial_\x \Psi|^2 \overline{\Psi} \partial_\x^2 \Psi - \frac{221}{4} (\overline{\Psi})^3 (\partial_\x \Psi)^2 \partial_\x^2 \Psi - \frac{213}{4} |\Psi|^2 |\partial_\x \Psi|^4\\
& - \frac{101}{2} |\partial_\x \Psi|^2 (\overline{\Psi})^2 (\partial_\x \Psi)^2 + \frac{47}{16} |\Psi|^6 \Psi \partial_\x^2 \overline{\Psi} + \frac{35}{4} |\Psi|^6 \overline{\Psi} \partial_\x^2 \Psi + \frac{71}{16} |\Psi|^4 (\Psi)^2 (\partial_\x \overline{\Psi})^2\\
& + \frac{117}{4} |\Psi|^6 |\partial_\x \Psi|^2 + \frac{175}{8} |\Psi|^4 (\overline{\Psi})^2 (\partial_\x \Psi)^2 - \frac{7}{16} |\Psi|^{10}.
\end{align*}

The second step is to provide explicit expressions of the invariants $I_n(\Psi)$ associated to each conservation law $f_n(\Psi)$ for an arbitrary function $\Psi$ in the appropriate $X^k(\R)$ space. This raises some serious difficulties since the integrands are not in general integrable when $\Psi$ belongs to $X^k(\R)$. For instance, according to definition \eqref{densities}, the invariants $I_1(\Psi)$, $I_2(\Psi)$ and $I_3(\Psi)$ should be given by
\begin{equation}
I_1(\Psi) = \frac{1}{2} \int_\R \big( 1 - |\Psi|^2 \big), I_2(\Psi) = - \frac{1}{2} \int_\R \overline{\Psi} \partial_\x \Psi, \ {\rm and} \ I_3(\Psi) = - \frac{1}{2} \int_\R \overline{\Psi} \partial_\x^2 \Psi + \frac{1}{4} \int_\R \big( |\Psi|^4 - 1 \big).
\end{equation}
For an arbitrary function $\Psi$ in $X^k(\R)$, none of the above integrands belong to $L^1(\R)$. Some quantities like $\overline{\Psi} \partial_\x^2 \Psi$ can be handled using integration by parts. This is not possible for $1 - |\Psi|^2$ or $|\Psi|^4 - 1$, which do not involve derivatives. Even the quantity $\overline{\Psi} \partial_\x\Psi$ cannot be immediately treated by integration by parts. In particular, the renormalization process as used in formula \eqref{densities} is not sufficient to give a sense to the invariants $I_n(\Psi)$ in the spaces $X^k(\R)$.

When $n = 2 m + 1$ is an odd number, a simple way to remove this difficulty is to introduce linear combinations of the conservation laws. More precisely, we consider the integral quantities formally defined by
\begin{align}
\label{renorE1}
E_1(\Psi) & = \int_{\R} \Big( f_3(\Psi) + f_1(\Psi) + \frac{1}{4} \Big),\\
\label{renorE2}
E_2(\Psi) & = - \int_{\R} \Big( f_5(\Psi) + 3 f_3(\Psi) + \frac{3}{2} f_1(\Psi) + \frac{1}{4} \Big),\\
\label{renorE3}
E_3(\Psi) & = \int_{\R} \Big( f_7(\Psi) + 5 f_5(\Psi) + \frac{15}{2} f_3(\Psi) + \frac{5}{2} f_1(\Psi) + \frac{5}{16} \Big),
\end{align}
and
\begin{equation}
\begin{split}
\label{renorE4}
E_4(\Psi) & = - \int_{\R} \Big( f_9(\Psi) + 7 f_7(\Psi) + \frac{35}{2} f_5(\Psi) + \frac{35}{2} f_3(\Psi) + \frac{35}{8} f_1(\Psi) + \frac{7}{16} \Big).
\end{split}
\end{equation}
Setting $\eta \equiv 1 - |\Psi|^2$ as usual, formal integrations by parts lead to the expressions
\begin{align}
\label{E1}
E_1(\Psi) \equiv & E(\Psi) = \frac{1}{2} \int_{\R} |\partial_\x \Psi|^2 + \frac{1}{4} \int_{\R} \eta^2,\\
\label{E2}
E_2(\Psi) \equiv & \frac{1}{2} \int_{\R} |\partial_\x^2 \Psi|^2 - \frac{3}{2} \int_{\R} \eta |\partial_\x \Psi|^2 + \frac{1}{4} \int_{\R} (\partial_\x \eta)^2 - \frac{1}{4} \int_\R \eta^3,\\
\label{E3}
E_3(\Psi) \equiv & \frac{1}{2} \int_{\R} |\partial_\x^3 \Psi|^2 + \frac{1}{4} \int_{\R} |\partial_\x^2 \eta|^2 + \frac{5}{4} \int_{\R} |\partial_\x \Psi|^4 + \frac{5}{2} \int_{\R} \partial_\x^2 \eta |\partial_\x \Psi|^2 - \frac{5}{2} \int_{\R} \eta |\partial_\x^2 \Psi|^2\\
\notag - & \frac{5}{4} \int_\R \eta (\partial_\x \eta)^2 + \frac{15}{4} \int_\R \eta^2 |\partial_\x \Psi|^2 + \frac{5}{16} \int_\R \eta^4,
\end{align}
and
\begin{equation}
\begin{split}
\label{E4}
E_4(\Psi) \equiv & \frac{1}{2} \int_{\R} |\partial_\x^4 \Psi|^2 + \frac{1}{4} \int_{\R} |\partial_\x^3 \eta|^2 - \frac{7}{4} \int_\R \eta (\partial_\x^2 \eta)^2 - \frac{7}{2} \int_\R \eta |\partial_\x^3 \Psi|^2 + \frac{35}{8} \int_{\R} \eta^2 (\partial_\x \eta)^2\\
+ & \frac{35}{4} \int_\R \eta^2 |\partial_\x^2 \Psi|^2 - \frac{35}{4} \int_{\R} (\partial_\x \eta)^2 |\partial_\x \Psi|^2 - \frac{7}{2} \int_\R |\partial_\x \Psi|^2 |\partial_\x^2 \Psi|^2 - 7 \int_\R \partial_\x^2 \eta \langle \partial_\x \Psi, \partial_\x^3 \Psi\rangle\\
- & 7 \int_\R |\partial_\x \Psi|^2 \langle \partial_\x \Psi, \partial_\x^3 \Psi \rangle - \frac{35}{2} \int_\R \eta \partial_\x^2 \eta |\partial_\x \Psi|^2 - \frac{35}{4} \int_\R \eta^3 |\partial_\x \Psi|^2 - \frac{35}{4} \int_\R \eta |\partial_\x \Psi|^4\\
- & \frac{7}{16} \int_\R \eta^5.
\end{split}
\end{equation}
These expressions involve only integrable integrands, and therefore provide a rigorous definition of the corresponding integrals. We will refer to $E_k(\Psi)$ as the $k^{\rm th}$-order energy.

When $n = 2 m$, with $m \geq 2$, the same strategy can be applied to define the $k^{\rm th}$-order momentum. We first introduce the formal linear combinations of even conservation laws
\begin{align}
\label{renorP2}
P_2(\Psi) & = i \int_{\R} \Big( f_4(\Psi) + \frac{3}{2} f_2(\Psi) \Big),\\
\label{renorP3}
P_3(\Psi) & = - i \int_{\R} \Big( f_6(\Psi) + 5 f_4(\Psi) + 5 f_2(\Psi) \Big),
\end{align}
and
\begin{equation}
\begin{split}
\label{renorP4}
P_4(\Psi) & = i \int_{\R} \Big( f_8(\Psi) + 7 f_6(\Psi) + \frac{35}{2} f_4(\Psi) + \frac{105}{8} f_2(\Psi) \Big).
\end{split}
\end{equation}
After some integrations by parts, these expressions are transformed into the well-defined quantities
\begin{align}
\label{P2}
P_2(\Psi) & \equiv \frac{1}{2} \int_{\R} \langle i \partial_\x^2 \Psi, \partial_\x \Psi \rangle - \frac{3}{4} \int_{\R} \eta \langle i \partial_\x \Psi, \Psi \rangle,\\
\label{P3}
P_3(\Psi) & \equiv \frac{1}{2} \int_{\R} \langle i \partial_\x^3 \Psi, \partial_\x^2 \Psi \rangle - \frac{5}{2} \int_{\R} \eta \langle i \partial_\x^2 \Psi, \partial_\x \Psi \rangle + \frac{5}{4} \int_{\R} (\eta^2 + \eta) \langle i \partial_\x \Psi, \Psi \rangle,
\end{align}
and
\begin{equation}
\begin{split}
\label{P4}
P_4(\Psi) \equiv & \frac{1}{2} \int_{\R} \langle i \partial_\x^4 \Psi, \partial_\x^3 \Psi \rangle - \frac{7}{2} \int_{\R} \eta \langle i \partial_\x^3 \Psi, \partial_\x^2 \Psi \rangle + \frac{7}{2} \int_{\R} \partial_\x^2 \eta \langle i \partial_\x^2 \Psi, \partial_\x \Psi \rangle + \frac{7}{4} \int_{\R} |\partial_\x \Psi|^2 \langle i \partial_\x^2 \Psi, \partial_\x \Psi \rangle\\
+ & \frac{35}{4} \int_{\R} \eta^2 \langle i \partial_\x^2 \Psi, \partial_\x \Psi \rangle - \frac{35}{16} \int_{\R} (\eta^3 + \eta^2 + \eta) \langle i \partial_\x \Psi, \Psi \rangle.
\end{split}
\end{equation}
 
The case $n = 2$ has to be discussed separately. The invariant $I_2(\Psi)$ is formally equal, up to some integration by parts, to
$$I_2(\Psi) = \frac{1}{4} \int_\R \Big( \Psi \partial_\x \overline{\Psi} - \overline{\Psi} \partial_\x \Psi \Big).$$
This quantity is purely imaginary. Its imaginary part is equal to the momentum, i.e.
\begin{equation}
\label{P1}
\Im(I_2(\Psi)) = P_1(\Psi) \equiv P(\Psi) = \frac{1}{2} \int_{\R} \langle i \partial_\x \Psi, \Psi \rangle.
\end{equation}
However, the definition of the momentum raises some difficulty. As a matter of fact, the quantity $P(\Psi)$ is not well-defined for any arbitrary map $\Psi$ in the energy space $X^1(\R)$. We refer to \cite{BeGrSaS1} for a proof of this claim, and a discussion about the different ways to provide a rigorous definition of the momentum in the energy space. Notice that in our analysis of the transonic limit, we handle with maps $\Psi$ with small energy. In particular, we may assume that they satisfy
\begin{equation}
\label{toutpetit}
E(\Psi) < \frac{2 \sqrt{2}}{3},
\end{equation}
so that we may lift $\Psi$ as
\begin{equation}
\label{lift}
\Psi = \varrho \exp i \varphi.
\end{equation}
Then, we may define a so-called renormalized momentum by
\begin{equation}
\label{p1}
p_1(\Psi) = p(\Psi) \equiv \frac{1}{2} \int_{\R} \eta \partial_\x \varphi
\end{equation}
(see \cite{BetGrSa2,BeGrSaS1} for more details), which is also, at least formally, an invariant for the Gross-Pitaevskii equation, since it verifies
\begin{equation}
\label{renorp1}
p_1(\Psi) = - i \int_{\R} f_2(\Psi),
\end{equation}
when $\Psi$ is sufficiently smooth and integrable at infinity.

We will also consider the renormalized momenta $p_k$, which are linear combinations of $P_k$ and $p_1$. They are defined by
\begin{align}
\label{p2}
p_2(\Psi) \equiv P_2(\Psi) - \frac{3}{2} p_1(\Psi) = & \frac{1}{2} \int_{\R} \langle i \partial_\x^2 \Psi, \partial_\x \Psi \rangle - \frac{3}{4} \int_{\R} \eta \langle i \partial_\x \Psi, \Psi \rangle - \frac{3}{4} \int_{\R} \eta \partial_\x \varphi,\\
\label{p3}
p_3(\Psi) \equiv P_3(\Psi) + \frac{5}{2} p_1(\Psi) = & \frac{1}{2} \int_{\R} \langle i \partial_\x^3 \Psi, \partial_\x^2 \Psi \rangle - \frac{5}{2} \int_{\R} (\eta - 1) \langle i \partial_\x^2 \Psi, \partial_\x \Psi \rangle\\
\notag & + \frac{5}{4} \int_{\R} (\eta^2 + \eta) \langle i \partial_\x \Psi, \Psi \rangle + \frac{5}{4} \int_{\R} \eta \partial_\x \varphi,
\end{align}
and
\begin{equation}
\begin{split}
\label{p4}
p_4(\Psi) \equiv P_4(\Psi) - \frac{35}{8} p_1(\Psi) = & \frac{1}{2} \int_{\R} \langle i \partial_\x^4 \Psi, \partial_\x^3 \Psi \rangle - \frac{7}{2} \int_{\R} \eta \langle i \partial_\x^3 \Psi, \partial_\x^2 \Psi \rangle + \frac{35}{4} \int_{\R} \eta^2 \langle i \partial_\x^2 \Psi, \partial_\x \Psi \rangle\\
& - \frac{35}{16} \int_{\R} (\eta^3 + \eta^2 + \eta) \langle i \partial_\x \Psi, \Psi \rangle - \frac{35}{16} \int_{\R} \eta \partial_\x \varphi,
\end{split}
\end{equation}
provided that the function $\Psi$ satisfies condition \eqref{toutpetit}. As a matter of fact, the renormalized momenta $p_k$, more than the momenta $P_k$, will be involved in the analysis of the transonic limit. 

We may summarize some of our previous discussion in 

\begin{lemma}
\label{DefEkPk}
The functionals $E_k$, for $1 \leq k \leq 4$, and $P_k$, for $2\leq k \leq 4$, are well-defined and continuous on $X^k(\R)$. The functionals $p_k(\Psi)$ are well-defined for any function $\Psi \in X^k(\R)$ which satisfies \eqref{toutpetit}.
\end{lemma}

\begin{proof}
The proof follows from the definition of the space $X^1(\R)$ for the functional $E_1 = E$. For the momentum $p_1 = p$, it is proved in \cite{BetGrSa2} that any function $\Psi \in X^1(\R)$ such that \eqref{toutpetit} holds, verifies
$$\rho_{\min} = \underset{x \in \R}{\inf} |\Psi(x)| > 0,$$
so that, denoting $\Psi = \varrho \exp i \varphi$ as above,
$$|\eta \partial_\x \varphi| \leq \frac{1}{\rho_{\min}} \big| \eta \big| \big| \varrho \partial_\x \varphi \big| \leq \frac{1}{\rho_{\min}} \big| \eta \big| \big| \partial_\x \Psi \big|.$$
Hence, the quantity $\eta \partial_\x \varphi$ belongs to $L^1(\R)$, so that $p(\Psi)$ is well-defined as well. Finally, for the higher order invariants, notice that, by the Sobolev embedding theorem, any function $\Psi \in X^k(\R)$ belongs to $\boC^{k-1}_0(\R)$, so that, in particular, $\eta$ is in $H^k(\R)$. Continuity raises no difficulty.
\end{proof}

\subsection{Conservation of the invariants in the spaces $X^k(\R)$}
\label{Definv}

The purpose of this section is to provide a rigorous mathematical proof to the fact that the invariants are conserved along the Gross-Pitaevskii flow. As mentioned in the introduction, conservation of the energy $E_1 = E$ was already addressed in \cite{Zhidkov1} (see also \cite{Gerard2}).

\begin{theorem}[\cite{Zhidkov1,Gerard2}]
\label{ConsE}
Let $\Psi_0 \in X^1(\R)$. Then, the unique solution $\Psi(\cdot, t)$ to \eqref{GP} in $\boC^0(\R, X^1(\R))$ with initial data $\Psi_0$ given by Theorem \ref{thm:existe} satisfies
$$E \big( \Psi(\cdot, t) \big) = E(\Psi_0),$$
for any $t \in \R$.
\end{theorem}

Concerning the momentum, Gallo \cite{Gallo3} established the conservation of the renormalized momentum $p_1$ (see also \cite{BeGrSaS1}).

\begin{theorem}[\cite{Gallo3, BeGrSaS1}]
\label{Defandconsp}
Let $\Psi_0$ be a function in $X^1(\R)$ which satisfies \eqref{toutpetit}. If $\Psi(\cdot, t)$ stands for the unique solution to \eqref{GP} in $\boC^0(\R, X^1(\R))$ with initial data $\Psi_0$ given by Theorem \ref{thm:existe} , then
$$p \big( \Psi(\cdot, t) \big) = p(\Psi_0),$$
for any $t \in \R$.
\end{theorem}

Here, we extend the analysis to the integral quantities $P_k(\Psi)$ and $E_k(\Psi)$.

\begin{theorem}
\label{Invconserved}
Let $2 \leq k \leq 4$ and $\Psi_0 \in X^k(\R)$. Then, the unique solution $\Psi(\cdot, t)$ in the space $\boC^0(\R, X^k(\R))$ to \eqref{GP} with initial data $\Psi_0$ given by Theorem \ref{thm:existe} satisfies
\begin{equation}
\label{buffet}
P_k \big( \Psi(\cdot, t) \big) = P_k(\Psi_0), \ {\rm and} \ E_k \big( \Psi(\cdot, t) \big) = E_k(\Psi_0),
\end{equation}
for any $t \in \R$.
\end{theorem}

\begin{remark}
Theorem \ref{Invconserved} focuses on the conservation of integral quantities which play a role in the analysis of the transonic limit. As mentioned in the introduction, the mass $m(\Psi)$ defined by \eqref{alamasse} is also formally conserved. However, the quantity $m(\Psi)$ is not well-defined in the energy space $X^1(\R)$. A proof of its conservation along the Gross-Pitaevskii flow would first require to provide a precise mathematical meaning to this quantity in $X^1(\R)$.\\
Similarly, Theorem \ref{Invconserved} does not address the question of the existence and conservation of higher order energies and momenta. A more general treatment of the inductive form of the conservation laws $f_n$ would be required to define properly higher order energies and momenta. However, we believe that such integral quantities could be well-defined in the spaces $X^k(\R)$ taking linear combinations and integrating by parts as above, so that their conservation along the Gross-Pitaevskii flow would also follow from Lemma \ref{Consfn} below.
\end{remark}

At this stage, notice that, in view of Theorems \ref{Defandconsp} and \ref{Invconserved}, and definitions \eqref{p2}, \eqref{p3} and \eqref{p4}, the quantities $p_k$ are also conserved along the Gross-Pitaevskii flow.

\begin{cor}
\label{Conspk}
Let $2 \leq k \leq 4$, and let $\Psi_0$ be a function in $X^k(\R)$ such that assumption \eqref{toutpetit} holds. Then, we have 
$$p_k \big( \Psi(\cdot, t) \big) = p_k(\Psi_0),$$
for any $t \in \R$, where $\Psi$ denotes the unique solution to \eqref{GP} in $\boC^0(\R, X^k(\R))$ with initial data $\Psi_0$.
\end{cor}

In the proof of Theorem \ref{Invconserved}, we will make use of the fact that the functionals $f_n$ are conservation laws for \eqref{GP}. More precisely, we have

\begin{lemma}
\label{Consfn}
Let $- \infty \leq a < b \leq + \infty$ and $n \geq 1$. Consider a solution $\Psi$ to \eqref{GP} such that
\begin{equation}
\label{aig}
\Psi \in \boC^0((a, b), \boC^{n+1}(\R)) \cap \boC^1((a, b), \boC^{n-1}(\R)).
\end{equation}
Then, the map $t \mapsto f_n(\Psi(\cdot, t))$ is in $\boC^0((a, b), \boC^1(\R)) \cap \boC^1((a, b), \boC^0(\R))$, while the function $t \mapsto f_{n+1}(\Psi(\cdot, t))$ belongs to $\boC^0((a, b), \boC^1(\R))$. Moreover, they satisfy
\begin{equation}
\label{lehmann}
\partial_t \Big( f_n(\Psi) \Big) = i \partial_\x \Big( f_{n + 1}(\Psi) - \partial_\x \overline{\Psi} \boF_n(\Psi) \Big) \ {\rm on} \ \R \times (a, b).
\end{equation}
\end{lemma}

We will first consider the maps $\boF_n(\Psi)$ defined by \eqref{recurFn}, and prove

\begin{lemma}
\label{ConsFn}
Let $- \infty \leq a < b \leq + \infty$ and $n \geq 1$. Consider a solution $\Psi$ to \eqref{GP} which satisfies \eqref{aig}. Then, the map $t \mapsto \boF_n(\Psi(\cdot, t))$ is in $\boC^0((a, b), \boC^1(\R)) \cap \boC^1((a, b), \boC^0(\R))$, while the function $t \mapsto \boF_{n+1}(\Psi(\cdot, t))$ belongs to $\boC^0((a, b), \boC^1(\R))$. Moreover, they satisfy
\begin{equation}
\label{nomura}
\partial_t \big( \boF_n(\Psi) \big) = i \boF_n(\Psi) - i |\Psi|^2 \boF_n(\Psi) + i \partial_\x \overline{\Psi} \sum_{j = 1}^{n - 1} \boF_j(\Psi) \boF_{n - j}(\Psi) + i \partial_\x \big( \boF_{n+1}(\Psi) \big).
\end{equation}
\end{lemma}

Lemma \ref{Consfn} is then a direct consequence of Lemma \ref{ConsFn}.

\begin{proof}[Proof of Lemma \ref{Consfn}]
Notice first that, in view of assumption \eqref{aig} and formulae \eqref{recurfn} and \eqref{recurFn}, the maps $f_j(\Psi)$ and $\boF_j(\Psi)$ belong to $\boC^0((a, b), \boC^1(\R))$ for any $1 \leq j \leq n + 1$, and the functionals $f_n(\Psi)$ and $\boF_n(\Psi)$ are also in $\boC^1((a, b), \boC^0(\R))$. Therefore, in view of \eqref{formfn}, we can write
$$\partial_t \big( f_n(\Psi) \big) = \partial_t \overline{\Psi} \boF_n(\Psi) + \overline{\Psi} \partial_t \big( \boF_n(\Psi) \big),$$
so that, by \eqref{GP} and \eqref{nomura},
$$\partial_t \big( f_n(\Psi) \big) = i \Big( - \partial_\x^2 \overline{\Psi} \boF_n(\Psi) + \overline{\Psi} \partial_\x \overline{\Psi} \sum_{j = 1}^{n - 1} \boF_j(\Psi) \boF_{n - j}(\Psi) + \overline{\Psi} \partial_\x \big( \boF_{n+1}(\Psi) \big) \Big).$$
In view of \eqref{recurFn}, we are led to
$$\partial_t \big( f_n(\Psi) \big) = i \Big( - \partial_\x^2 \overline{\Psi} \boF_n(\Psi) + \partial_\x \overline{\Psi} \boF_{n+1}(\Psi) - \partial_\x \overline{\Psi} \partial_\x \big( \boF_n(\Psi) \big) + \overline{\Psi} \partial_\x \big( \boF_{n+1}(\Psi) \big) \Big),$$
which completes the proof of \eqref{lehmann}, invoking definition \eqref{formfn}.
\end{proof}

We now provide the proof of Lemma \ref{ConsFn}.

\begin{proof}[Proof of Lemma \ref{ConsFn}]
The proof is by induction on $n \in \N^*$. For $n = 1$, it follows from \eqref{GP} that
\begin{align}
\partial_t \big( \boF_1(\Psi) \big) = - \frac{1}{2} \partial_t \Psi = \frac{i}{2} \Big( - \partial_\x^2 \Psi - \Psi + |\Psi|^2 \Psi \Big) = i \boF_1(\Psi) - i |\Psi|^2 \boF_1(\Psi) + i \partial_\x \big( \boF_2(\Psi) \big),
\end{align}
so that \eqref{nomura} holds for $n = 1$. We now turn to the case $n = N + 1$, assuming that the conclusion of Lemma \ref{Consfn} holds for any $1 \leq n \leq N$.
Notice first that, in view of assumption \eqref{aig} and formulae \eqref{recurfn} and \eqref{recurFn}, the maps $\boF_j(\Psi)$ are in $\boC^0((a, b), \boC^1(\R))$ for any $1 \leq j \leq N + 2$, while the functional $\boF_{N + 1}(\Psi)$ also belongs to $\boC^1((a, b), \boC^0(\R))$. Therefore, in view of \eqref{recurFn}, we can write
\begin{align*}
\partial_t \big( \boF_{N+1}(\Psi) \big) = & \partial_t \partial_\x \big( \boF_N(\Psi) \big) + \partial_t \overline{\Psi} \sum_{j = 1}^{N - 1} \boF_j(\Psi) \boF_{N-j}(\Psi) \\ & + \overline{\Psi} \sum_{j = 1}^{N - 1} \Big( \partial_t \big( \boF_j(\Psi) \big) \boF_{N-j}(\Psi) + \boF_j(\Psi) \big( \partial_t \boF_{N-j}(\Psi) \big) \Big).
\end{align*}
Invoking the inductive assumption combined with \eqref{GP}, we are led to
\begin{equation}
\label{stanley}
\begin{split}
\partial_t \big( \boF_{N+1}(\Psi) \big) = & i \bigg( \big( 1 - |\Psi|^2 \big) \partial_\x \big( \boF_N(\Psi) \big) - \big( \Psi \partial_\x \overline{\Psi} + \overline{\Psi} \partial_\x \Psi \big) \boF_N(\Psi)\\
+ & \partial_\x^2 \big( \boF_{N+1}(\Psi) \big) + \overline{\Psi} \big( 1 - |\Psi|^2 \big) \sum_{j = 1}^{N - 1} \boF_j(\Psi) \boF_{N-j}(\Psi)\\
+ & \partial_\x \overline{\Psi} \sum_{j = 1}^{N - 1} \Big( \partial_\x \big( \boF_j(\Psi) \big) \boF_{N-j}(\Psi) + \boF_j(\Psi) \partial_\x \big( \boF_{N-j}(\Psi) \big) \Big)\\
+ & \overline{\Psi} \sum_{j = 1}^{N - 1} \Big( \partial_\x \big( \boF_{j+1}(\Psi) \big) \boF_{N-j}(\Psi) + \boF_j(\Psi) \partial_\x \big( \boF_{N+1-j}(\Psi) \big) \Big)\\
+ & \overline{\Psi} \partial_\x \overline{\Psi} \sum_{j = 1}^{N - 1} \Big( \boF_{N-j}(\Psi) \sum_{k = 1}^{j - 1} \boF_k(\Psi) \boF_{j-k}(\Psi) + \boF_j(\Psi) \sum_{k = 1}^{N - j - 1} \boF_k(\Psi) \boF_{N-j-k}(\Psi) \Big) \bigg).
\end{split}
\end{equation}
In view of \eqref{boF1}, we first have
\begin{equation}
\begin{split}
\label{cdo1}
- \overline{\Psi} \partial_\x \Psi \boF_N(\Psi) & + \overline{\Psi} \sum_{j = 1}^{N - 1} \Big( \partial_\x \big( \boF_{j+1}(\Psi) \big) \boF_{N-j}(\Psi) + \boF_j(\Psi) \partial_\x \big( \boF_{N+1-j}(\Psi) \big) \Big)\\
= & \overline{\Psi} \sum_{j = 1}^N \Big( \partial_\x \big( \boF_j(\Psi) \big) \boF_{N+1-j}(\Psi) + \boF_j(\Psi) \partial_\x \big( \boF_{N+1-j}(\Psi) \big) \Big),
\end{split}
\end{equation}
whereas, by formula \eqref{recurFn},
\begin{equation}
\label{cdo2}
\big( 1 - |\Psi|^2 \big) \Big( \partial_\x \big( \boF_N(\Psi) \big) + \overline{\Psi} \sum_{j = 1}^{N - 1} \boF_j(\Psi) \boF_{N-j}(\Psi) \Big) = \big( 1 - |\Psi|^2 \big) \boF_{N+1}(\Psi),
\end{equation}
and
\begin{equation}
\begin{split}
\label{cdo3}
- \Psi \partial_\x \overline{\Psi} & \boF_N(\Psi) + \overline{\Psi} \partial_\x \overline{\Psi} \sum_{j = 1}^{N - 1} \Big( \boF_{N-j}(\Psi) \sum_{k = 1}^{j - 1} \boF_k(\Psi) \boF_{j-k}(\Psi) + \boF_j(\Psi) \sum_{k = 1}^{N - j - 1} \boF_k(\Psi) \boF_{N-j-k}(\Psi) \Big)\\
= & 2 \partial_\x \overline{\Psi} \sum_{j = 1}^N \boF_j(\Psi) \boF_{N+1-j}(\Psi) - \partial_\x \overline{\Psi} \sum_{j = 1}^{N - 1} \Big( \partial_\x \big( \boF_j(\Psi) \big) \boF_{N-j}(\Psi) + \boF_j(\Psi) \partial_\x \big( \boF_{N-j}(\Psi) \big) \Big).
\end{split}
\end{equation}
Hence, we deduce from \eqref{stanley}, \eqref{cdo1}, \eqref{cdo2} and \eqref{cdo3} that
\begin{equation}
\label{morgan}
\begin{split}
\partial_t \big( \boF_{N+1}(\Psi) \big) = i \bigg( \big( 1 - & |\Psi|^2 \big) \boF_{N+1}(\Psi) + \partial_\x \overline{\Psi} \sum_{j = 1}^N \boF_j(\Psi) \boF_{N+1-j}(\Psi)\\
+ & \partial_\x \Big( \partial_\x \big( \boF_{N+1}(\Psi) \big) + \overline{\Psi} \sum_{j = 1}^N \boF_j(\Psi) \boF_{N+1-j}(\Psi) \Big) \bigg).
\end{split}
\end{equation}
In view of \eqref{recurFn}, the second line in \eqref{morgan} is equal to $\partial_\x \big( \boF_{N+2}(\Psi) \big)$, so that \eqref{nomura} holds for $n = N + 1$. This completes the proof of Lemma \ref{ConsFn} by induction.
\end{proof}

We finally turn to the proof of Theorem \ref{Invconserved}

\begin{proof}[Proof of Theorem \ref{Invconserved}]
We first assume that in addition $\Psi_0 \in X^9(\R)$. In this situation, the maps $t \mapsto E_k(\Psi(\cdot, t))$ and $t \mapsto P_k(\Psi(\cdot, t))$ are in $\boC^1(\R, \R)$, while by the Sobolev embedding theorem, the map $t \mapsto \Psi(\cdot, t)$ is in $\boC^1(\R, \boC^8(\R))$ and $\boC^0(\R, \boC^{10}(\R))$. Hence, in view of Lemma \ref{Consfn},
\begin{equation}
\label{mitsu}
\partial_t \Big( f_n(\Psi) \Big) = i \partial_\x \Big( f_{n + 1}(\Psi) - \partial_\x \overline{\Psi} \boF_n(\Psi) \Big) \ {\rm on} \ \R,
\end{equation}
for any $1 \leq n \leq 9$.

Now consider, for instance, the map $t \mapsto E_2(\Psi(\cdot, t))$. In view of \eqref{renorE2}, its derivative is, at least formally, given by
$$\frac{\rm d}{\rm dt} \Big( E_2(\Psi(\cdot, t)) \Big) = - \int_{\R} \Big( \partial_t f_5(\Psi) + 3 \partial_t f_3(\Psi) + \frac{3}{2} \partial_t f_1(\Psi) \Big),$$
so that by \eqref{mitsu}, we formally have
$$\frac{\rm d}{\rm dt} \Big( E_2(\Psi(\cdot, t)) \Big) = - i \int_{\R} \partial_\x \bigg( f_6(\Psi) + 3 f_4(\Psi) + \frac{3}{2} f_2(\Psi) - \partial_\x \overline{\Psi} \Big( \boF_5(\Psi) + 3 \boF_2(\Psi) + \frac{3}{2} \boF_1(\Psi) \Big) \bigg) = 0,$$
i.e. the quantity $E_2(\Psi)$ is formally conserved by \eqref{GP}. In particular, the proof of the conservation of $E_2$ along the Gross-Pitaevskii flow reduces to drop some integrability difficulties in the above formal argument. Therefore, given any $R > 1$, we introduce some cut-off function $\chi \in \boC^\infty(\R, [0, 1])$ such that
\begin{equation}
\label{plouto}
\chi = 1 \ {\rm on} \ (-1, 1), \ {\rm and} \ \chi = 0 \ {\rm on} \ \R \setminus (-2, 2),
\end{equation}
and denote
\begin{equation}
\label{dingo}
\chi_R(\x) = \chi \Big( \frac{\x}{R} \Big), \ \forall \x \in \R.
\end{equation}
Since the map $t \mapsto \Psi(\cdot, t)$ belongs to $\boC^1(\R, X^9(\R))$, we then have
\begin{equation}
\label{mickey}
\frac{\rm d}{\rm dt} \Big( E_2(\Psi(\cdot, t)) \Big) = \int_\R \partial_t \big( e_2(\Psi(\cdot,t)) \big) = \underset{R \to + \infty}{\lim} \int_\R \chi_R(\x) \partial_t \big( e_2(\Psi(\x,t)) \big) d\x,
\end{equation}
where we let
$$E_2(\psi) \equiv \int_\R e_2(\psi).$$
We now make use of formal relation \eqref{renorE2} to compute
$$\int_\R \chi_R(\x) \partial_t \big( e_2(\Psi(\x, t)) \big) d\x = - \int_{\R} \chi_R \Big( \partial_t f_5(\Psi) + 3 \partial_t f_3(\Psi) + \frac{3}{2} \partial_t f_1(\Psi) \Big) + \int_{\R} \partial_\x \chi_R \ Q_1(\Psi, \partial_t \Psi),$$
where, using definitions \eqref{f1}, \eqref{f3}, \eqref{f5} and \eqref{E2}, and the Sobolev embedding theorem, the function $Q_1(\Psi, \partial_t \Psi)$ tends to $0$ at $\pm \infty$. Invoking \eqref{mitsu} and integrating by parts once more, we are led to
\begin{equation}
\label{minnie}
\int_\R \chi_R(\x) \partial_t \big( e_2(\Psi(\x, t)) \big) d\x = \int_{\R} \partial_\x \chi_R \ Q_2(\Psi, \partial_t \Psi),
\end{equation}
where
$$Q_2(\Psi, \partial_t \Psi) = Q_1(\Psi, \partial_t \Psi) + i f_6(\Psi) + 3 i f_4(\Psi) + \frac{3}{2} i f_2(\Psi) - i \partial_\x \overline{\Psi} \Big( \boF_5(\Psi) + 3 \boF_2(\Psi) + \frac{3}{2} \boF_1(\Psi) \Big),$$
also tends to $0$ at $\pm \infty$. Finally, notice that when
$$f(\x) \to 0, \ {\rm as} \ |\x| \to + \infty,$$
we have
$$\int_{\R} \partial_\x^j \chi_R(\x) f(\x) d\x = \frac{1}{R^{j-1}} \bigg( \int_1^2 \partial_\x^j \chi(\x) f(R \x) d\x + \int_{- 2}^{- 1} \partial_\x^j \chi(\x) f(R \x) d\x \bigg) \to 0, \ {\rm as} \ R \to + \infty,$$
so that, in view of \eqref{mickey} and \eqref{minnie}, we obtain at the limit $R \to + \infty$,
$$\frac{\rm d}{\rm dt} \Big( E_2(\Psi(\cdot, t)) \Big) = 0,$$
which gives \eqref{buffet} for the quantity $E_2$.

Using formal identities \eqref{renorE3}, \eqref{renorE4}, \eqref{renorP2}, \eqref{renorP3} and \eqref{renorP4}, the proofs are identical for the functionals $E_3$, $E_4$, $P_2$, $P_3$ and $P_4$, so that we omit them.

In the general case where we only have $\Psi_0 \in X^k(\R)$, we first approximate $\Psi_0$ by a sequence of functions $\psi_n$ in $X^9(\R)$ for the $X^k$-distance (see e.g. \cite{Gerard1}), and then use the continuity of the flow map $\Psi_0 \mapsto \Psi(\cdot, T)$ in $X^k(\R)$ for any fixed $T$, and the continuity of the functionals $E_k$ and $p_k$ with respect to the $X^k$-distance.
\end{proof}

\section{Invariants in the transonic limit}
\label{Rescaledinv}

In this section, we analyse the expressions of the invariant quantities introduced in the previous section in the slow variables. Therefore, we introduce the quantities $\boE_k(N_\eps, \Theta_\eps)$ and $\boP_k(N_\eps, \Theta_\eps)$ defined by
\begin{equation}
\label{slowEk}
E_k(\Psi) = \frac{\varepsilon^{2 k + 1}}{18} \boE_k(N_\eps, \Theta_\eps),
\end{equation}
and
\begin{equation}
\label{slowpk}
\begin{split}
p_k(\Psi) = & \frac{\varepsilon^{2 k + 1}}{18} \boP_k(N_\eps, \Theta_\eps).
\end{split}
\end{equation}
We also set
$$m_\eps = 1 - \frac{\eps^2}{6} N_\eps.$$
We now derive the precise expansions of $\boE_k$ and $\boP_k$ and stress the relationship with the corresponding \eqref{KdV} invariants.

\subsection{Formulae of the invariants in the rescaled variables}
\label{Squale}

For the $k^{th}$-energies defined by \eqref{E2}, \eqref{E3} and \eqref{E4}, a direct computation provides, in view of definitions \eqref{slow-var} of $N_\varepsilon$ and $\Theta_\varepsilon$,

\begin{lemma}
\label{Esquale}
Let $2 \leq k \leq 4$ and $\varepsilon > 0$. Given any function $\Psi$ in $X^k(\R)$ which satisfies \eqref{toutpetit}, and denoting $N_\varepsilon$ and $\Theta_\varepsilon$, the functions defined by \eqref{slow-var}, we have
\begin{equation}
\begin{split}
\label{boE2}
\boE_2(N_\eps, \Theta_\eps) & \equiv \frac{1}{8} \int_\R \bigg( \big( \partial_x N_\eps \big)^2 + m_\eps \Big( \partial_x^2 \Theta_\eps - \frac{\eps^2}{6 m_\eps} \partial_x N_\eps \partial_x \Theta_\eps \Big)^2 - \frac{1}{6} N_\eps^3 - \frac{m_\eps}{2} N_\eps \big( \partial_x \Theta_\eps \big)^2 \bigg)\\
& + \frac{\eps^2}{16} \int_\R \frac{1}{m_\eps} \bigg( \partial_x^2 N_\eps + \frac{m_\eps}{6} \big( \partial_x \Theta_\eps \big)^2 + \frac{\eps^2}{12 m_\eps} (\partial_x N_\eps)^2 \bigg)^2 - \frac{\eps^2}{32} \int_\R \boR_2(N_\eps, \Theta_\eps),
\end{split}
\end{equation}
with
\begin{equation}
\label{R2eps}
\boR_2(N_\eps, \Theta_\eps) \equiv \frac{N_\eps (\partial_x N_\eps)^2}{m_\eps},
\end{equation}
\begin{equation}
\begin{split}
\label{boE3}
\boE_3(N_\eps, \Theta_\eps) = & \frac{1}{8} \int_\R \bigg( m_\eps \Big( \partial_x^3 \Theta_\eps - \frac{\eps^2}{72} \big( \partial_x \Theta_\eps \big)^3 - \frac{\eps^2 \partial_x N_\eps \partial_x^2 \Theta_\eps}{4 m_\eps} - \frac{\eps^2 \partial_x^2 N_\eps \partial_x \Theta_\eps}{4 m_\eps} - \frac{\eps^4 (\partial_x N_\eps)^2 \partial_x \Theta_\eps}{48 m_\eps^2} \Big)^2\\
& + \big( \partial_x^2 N_\eps \big)^2 - \frac{5}{6} \Big( 2 \partial_x N_\eps \partial_x \Theta_\eps \partial_x^2 \Theta_\eps + N_\eps \big( \partial_x^2 \Theta_\eps \big)^2 + N_\eps \big( \partial_x N_\eps \big)^2 \Big) + \frac{5}{144} \Big( \big( \partial_x \Theta_\eps \big)^4\\
& + 6 N_\eps^2 \big( \partial_x \Theta_\eps \big)^2 + N_\eps^4 \Big) \bigg) + \frac{\eps^2}{16} \int_\R \frac{1}{m_\eps} \bigg( \partial_x^3 N_\eps - \frac{\eps^2}{24} \partial_x N_\eps \big( \partial_x \Theta_\eps \big)^2 + \frac{m_\eps}{2} \partial_x \Theta_\eps \partial_x^2 \Theta_\eps\\
& + \frac{\eps^2}{4 m_\eps} \partial_x N_\eps \partial_x^2 N_\eps + \frac{\eps^4}{48 m_\eps^2} \big( \partial_x N_\eps \big)^3 \bigg)^2 + \frac{\eps^2}{96} \int_\R \boR_3(N_\eps, \Theta_\eps),
\end{split}
\end{equation}
where
\begin{equation}
\begin{split}
\label{R3eps}
& \boR_3(N_\eps, \Theta_\eps) = \frac{5}{3} N_\eps^2 \big( \partial_x^2 \Theta_\eps \big)^2 - \frac{5}{4} \big( \partial_x N_\eps \big)^2 \big( \partial_x \Theta_\eps \big)^2 - 5 N_\eps \partial_x^2 N_\eps \big( \partial_x \Theta_\eps \big)^2 - \frac{5}{12} N_\eps^3 \big( \partial_x \Theta_\eps \big)^2\\
& - \frac{5}{18} N_\eps \big( \partial_x \Theta_\eps \big)^4 - \frac{5}{12} N_\eps^3 \big( \partial_x \Theta_\eps \big)^2 - \frac{5}{m_\eps} N_\eps \big( \partial_x^2 N_\eps \big)^2 + \frac{5}{4 m_\eps} N_\eps^2 \big( \partial_x N_\eps \big)^2 + \frac{\eps^2}{6} \Big( \frac{5}{24} N_\eps^2 \big( \partial_x \Theta_\eps \big)^4\\
& - \frac{5}{2 m_\eps} N_\eps \big( \partial_x N_\eps \big)^2 \big( \partial_x \Theta_\eps \big)^2 - \frac{25}{24 m_\eps^2} \big( \partial_x N_\eps \big)^4 - \frac{5}{m_\eps^2} N_\eps \big( \partial_x N_\eps \big)^2 \partial_x^2 N_\eps - \frac{5 \eps ^2}{24 m_\eps^3} N_\eps \big( \partial_x N_\eps \big)^4 \Big).
\end{split}
\end{equation}
and
\begin{equation}
\begin{split}
\label{boE4}
\boE_4 & (N_\eps, \Theta_\eps) = \frac{1}{8} \int_\R \bigg( m_\eps \Big( \partial_x^4 \Theta - \frac{\eps^2}{2 m_\eps} \partial_x^2 N_\eps \partial_x^2 \Theta_\eps - \frac{\eps^2}{3 m_\eps} \partial_x N_\eps \partial_x^3 \Theta_\eps - \frac{\eps^2}{3 m_\eps} \partial_x^3 N_\eps \partial_x \Theta_\eps\\
& - \frac{\eps^2}{12} \big( \partial_x \Theta_\eps \big)^2 \partial_x^2 \Theta_\eps - \frac{\eps^4}{24 m_\eps^2} \big( \partial_x N_\eps \big)^2 \partial_x^2 \Theta_\eps - \frac{\eps^4}{12 m_\eps^2} \partial_x N_\eps \partial_x^2 N_\eps \partial_x \Theta_\eps + \frac{\eps^4}{216 m_\eps} \partial_x N_\eps \big( \partial_x \Theta_\eps \big)^3\\
& - \frac{\eps^6}{144 m_\eps^3} \big( \partial_x N_\eps \big)^3 \partial_x \Theta_\eps \Big)^2 + \big( \partial_x^3 N_\eps \big)^2 - \frac{7}{6} \Big( N_\eps \big( \partial_x^2 N_\eps \big)^2 + m_\eps N_\eps \big( \partial_x^3 \Theta_\eps \big)^2\\
& + 2 m_\eps \partial_x^2 N_\eps \partial_x \Theta_\eps \partial_x^3 \Theta_\eps \Big) + \frac{35}{72} \Big( N_\eps^2 \big( \partial_x N_\eps \big)^2 + m_\eps \big( \partial_x \Theta_\eps \big)^2 \big( \partial_x^2 \Theta_\eps \big)^2 + 4 N_\eps \partial_x N_\eps \partial_x \Theta_\eps \partial_x^2 \Theta_\eps\\
& + m_\eps N_\eps^2 \big( \partial_x^2 \Theta_\eps \big)^2 + \big( \partial_x N_\eps \big)^2 \big( \partial_x \Theta_\eps \big)^2 \Big) -\frac{7}{864} \Big( N_\eps^5 + 5 N_\eps \big( \partial_x \Theta_\eps \big)^4 + 10 N_\eps^3 \big( \partial_x \Theta_\eps \big)^2 \Big) \bigg)\\
& + \frac{\eps^2}{16} \int_\R \frac{1}{m_\eps} \Big( \partial_x^4 N_\eps + \frac{m_\eps}{2} \big( \partial_x^2 \Theta_\eps \big)^2 + \frac{2}{3} m_\eps \partial_x \Theta_\eps \partial_x^3 \Theta_\eps + \frac{\eps^2}{3 m_\eps} \partial_x N_\eps \partial_x^3 N_\eps + \frac{\eps^2}{4 m_\eps} \big( \partial_x^2 N_\eps \big)^2\\
& - \frac{\eps^2}{12} \partial_x^2 N_\eps \big( \partial_x \Theta_\eps \big)^2 - \frac{\eps^2}{6} \partial_x N_\eps \partial_x \Theta_\eps \partial_x^2 \Theta_\eps - \frac{\eps^2}{432} m_\eps \big( \partial_x \Theta_\eps \big)^4 - \frac{\eps^4}{144 m_\eps} \big( \partial_x N_\eps \big)^2 \big( \partial_x \Theta_\eps \big)^2\\
& + \frac{\eps^4}{8 m_\eps^2} \big( \partial_x N_\eps \big)^2 \partial_x^2 N_\eps + \frac{5 \eps^6}{576 m_\eps^3} \big( \partial_x N_\eps \big)^4 \Big)^2 + \frac{\eps^2}{48} \int_\R \boR_4(N_\eps, \Theta_\eps),
\end{split}
\end{equation}
where
\begin{equation}
\begin{split}
\label{R4eps}
& \boR_4(N_\eps, \Theta_\eps) = \frac{35}{432} m_\eps N_\eps^2 \big( \partial_x \Theta_\eps \big)^4 + \frac{7}{864} m_\eps \big( \partial_x \Theta_\eps \big)^6 - \frac{35}{18} m_\eps N_\eps \big( \partial_x \Theta_\eps \big)^2 \big( \partial_x \Theta_\eps \big)^2\\
& - \frac{35}{9} N_\eps^2 \partial_x N_\eps \partial_x \Theta_\eps \partial_x^2 \Theta_\eps + \frac{35}{72} \partial_x^2 N_\eps \big( \partial_x \Theta_\eps \big)^4 - \frac{175}{72} N_\eps \big( \partial_x N_\eps \big)^2 \big( \partial_x \Theta_\eps \big)^2 - \frac{35}{24} \big( \partial_x N_\eps \big)^2 \big( \partial_x^2 \Theta_\eps \big)^2\\
& + \frac{91}{24} \big( \partial_x^2 N_\eps \big)^2 \big( \partial_x \Theta_\eps \big)^2 + \frac{35}{6} \partial_x N_\eps \partial_x^2 N_\eps \partial_x \Theta_\eps \partial_x^2 \Theta_\eps + \frac{21}{12} N_\eps \partial_x^2 N_\eps \big( \partial_x^2 \Theta_\eps \big)^2 + \frac{2}{3} N_\eps \partial_x^2 N_\eps \partial_x \Theta_\eps \partial_x^3 \Theta_\eps\\
& + \frac{7}{2} \big( \partial_x^2 N_\eps \big)^3 - \frac{35}{144 m_\eps} N_\eps^3 \big( \partial_x N_\eps \big)^2 - \frac{35}{72 m_\eps} \big( \partial_x N_\eps \big)^4 + \frac{35}{24 m_\eps} N_\eps^2 \big( \partial_x^2 N_\eps \big)^2 + \frac{\eps^2}{4} \Big( \frac{7}{2 m_\eps^2} N_\eps \big( \partial_x^2 N_\eps \big)^3\\
& - \frac{7 m_\eps}{46656} N_\eps \big( \partial_x \Theta_\eps \big)^6 - \frac{35}{54} N_\eps \partial_x^2 N_\eps \big( \partial_x \Theta_\eps \big)^4 - \frac{35}{432} \big( \partial_x N_\eps \big)^2 \big( \partial_x \Theta_\eps \big)^4 + \frac{35}{72 m_\eps} N_\eps^2 \big( \partial_x N_\eps \big)^2 \big( \partial_x \Theta_\eps \big)^2\\
& + \frac{7}{ 3 m_\eps} \big( \partial_x N_\eps \big)^2 \big( \partial_x^2 N_\eps \big)^2 - \frac{35}{12 m_\eps} N_\eps \big( \partial_x N_\eps \big)^2 \big( \partial_x^2 \Theta_\eps \big)^2 - \frac{35}{18 m_\eps} \big( \partial_x N_\eps \big)^2 \partial_x^2 N_\eps \big( \partial_x \Theta_\eps \big)^2\\
& - \frac{35}{12 m_\eps} N_\eps \big( \partial_x^2 N_\eps \big)^2 \big( \partial_x \Theta _\eps \big)^2 - \frac{35}{3 m_\eps} N_\eps \partial_x N_\eps \partial_x^2 N_\eps \partial_x \Theta_\eps \partial_x^2 \Theta_\eps - \frac{35}{54 m_\eps^2} N_\eps \big( \partial_x N_\eps \big)^4 \Big)\\
& + \frac{\eps^4}{144} \Big( \frac{35}{8 m_\eps^2} \big( \partial_x N_\eps \big)^4 \big( \partial_x \Theta_\eps \big)^2 - \frac{35}{96 m_\eps} N_\eps \big( \partial_x N_\eps \big)^2 \big( \partial_x \Theta_\eps \big)^4 + \frac{35}{m_\eps^2} N_\eps \big( \partial_x N_\eps \big)^2 \partial_x^2 N_\eps \big( \partial_x \Theta_\eps \big)^2\\
& - \frac{175}{72 m_\eps^3} N_\eps^2 \big( \partial_x N_\eps \big)^4 + \frac{147}{2 m_\eps^3} N_\eps \big( \partial_x N_\eps \big)^2 \big( \partial_x^2 N_\eps \big)^2 \Big) + \frac{\eps^6}{6912} \Big( \frac{245}{m_\eps^3} N_\eps \big( \partial_x N_\eps \big)^4 \big( \partial_x \Theta_\eps \big)^2\\
& - \frac{497}{5 m_\eps^4} \big( \partial_x N_\eps \big)^6 \Big) - \frac{7 \eps^8}{2560 m_\eps^5} N_\eps \big( \partial_x N_\eps \big)^6.
\end{split}
\end{equation}
\end{lemma}

Similarly, for the $k^{th}$-renormalized momenta, we compute

\begin{lemma}
\label{Psquale}
Let $2 \leq k \leq 4$ and $\varepsilon > 0$. Given any function $\Psi$ in $X^k(\R)$ which satisfies \eqref{toutpetit}, and denoting $N_\varepsilon$ and $\Theta_\varepsilon$, the functions defined by \eqref{slow-var}, we have
\begin{equation}
\label{boP2eps}
\boP_2(N_\eps, \Theta_\eps) \equiv \frac{1}{4 \sqrt{2}} \int_\R \bigg( \partial_x N_\eps \partial_x^2 \Theta_\eps - \frac{m_\eps}{12} \big( \partial_x \Theta_\eps \big)^3 - \frac{1}{4} N_\eps^2 \partial_x \Theta_\eps \bigg) - \frac{\eps^2}{32 \sqrt{2}} \int_\R r_2(N_\eps, \Theta_\eps),
\end{equation}
with
\begin{equation}
\label{r2eps}
r_2(N_\eps, \Theta_\eps) \equiv \frac{(\partial_x N_\eps)^2 \partial_x \Theta_\eps}{m_\eps},
\end{equation}
\begin{equation}
\label{boP3eps}
\begin{split}
\boP_3(N_\eps, \Theta_\eps) \equiv & \frac{1}{4 \sqrt{2}} \int_\R \bigg( \partial_x^2 N_\eps \partial_x^3 \Theta_\eps - \frac{5}{12} \Big( \partial_x \Theta_\eps \big( \partial_x^2 \Theta_\eps \big)^2 + 2 N_\eps \partial_x N_\eps \partial_x^2 \Theta_\eps + \big( \partial_x N_\eps \big)^2 \partial_x \Theta_\eps \Big)\\
& + \frac{5}{72} \Big( m_\eps N_\eps (\partial_x \Theta_\eps)^3 + N_\eps^3 \partial_x \Theta_\eps \Big) \bigg) + \frac{\eps^2}{48 \sqrt{2}} \int_\R r_3(N_\eps, \Theta_\eps),
\end{split}
\end{equation}
with
\begin{equation}
\begin{split}
\label{r3eps}
& r_3(N_\eps, \Theta_\eps) \equiv \frac{5}{6} N_\eps \partial_x \Theta_\eps \big( \partial_x^2 \Theta_\eps \big)^2 + \frac{5}{3} \partial_x N_\eps \big( \partial_x \Theta_\eps \big)^2 \partial_x^2 \Theta_\eps - \frac{m_\eps}{72} \big( \partial_x \Theta_\eps \big)^5 + \frac{5}{4 m_\eps} N_\eps \big( \partial_x N_\eps \big)^2 \partial_x \Theta_\eps\\
& - \frac{5}{m_\eps} \partial_x N_\eps \partial_x^2 N_\eps \partial_x^2 \Theta_\eps - \frac{5}{2 m_\eps} \big( \partial_x^2 N_\eps \big)^2 \partial_x \Theta_\eps - \frac{5 \eps^2}{72 m_\eps} \big( \partial_x N_\eps \big)^2 \big( \partial_x \Theta_\eps \big)^3 - \frac{5 \eps^2}{18 m_\eps^2} \big( \partial_x N_\eps\big)^3 \partial_x^2 \Theta_\eps\\
& + \frac{25 \eps^4}{864 m_\eps^3} \big( \partial_x N_\eps \big)^4 \partial_x \Theta_\eps,
\end{split}
\end{equation}
and
\begin{equation}
\begin{split}
\label{boP4eps}
\boP_4(N_\eps, \Theta_\eps) & \equiv \frac{1}{4 \sqrt{2}} \int_\R \bigg( \partial_x^3 N_\eps \partial_x^4 \Theta_\eps - \frac{7}{12} \Big( 2 N_\eps \partial_x^2 N_\eps \partial_x^3 \Theta_\eps + \big( \partial_x^2 N_\eps \big)^2 \partial_x \Theta_\eps + m_\eps \partial_x \Theta_\eps \big( \partial_x^3 \Theta_\eps \big)^2 \Big)\\
& + \frac{35}{72} \Big( N_\eps^2 \partial_x N_\eps \partial_x^2 \Theta_\eps + N_\eps \big( \partial_x N_\eps \big)^2 \partial_x \Theta_\eps + \partial_x N_\eps \big( \partial_x \Theta_\eps \big)^2 \partial_x^2 \Theta_\eps + m_\eps N_\eps \partial_x \Theta_\eps \big( \partial_x^2 \Theta_\eps \big)^2 \Big)\\
& - \frac{7}{1728} \Big( \big( \partial_x \Theta_\eps \big)^5 + 10 m_\eps N_\eps^2 \big( \partial_x \Theta_\eps \big)^3 + 5 N_\eps^4 \partial_x \Theta_\eps \Big) \bigg) + \frac{\eps^2}{48 \sqrt{2}} \int_\R r_4(N_\eps, \Theta_\eps),
\end{split}
\end{equation}
with
\begin{equation}
\begin{split}
\label{r4eps}
& r_4(N_\eps, \Theta_\eps) \equiv 7 \Big( \frac{5}{12 m_\eps} N_\eps \big( \partial_x^2 N_\eps \big)^2 \partial_x \Theta_\eps + \frac{5}{6 m_\eps} N_\eps \partial_x N_\eps \partial_x^2 N_\eps \partial_x^2 \Theta_\eps - \frac{5}{48 m_\eps} N_\eps^2 \big( \partial_x N_\eps \big)^2 \partial_x \Theta_\eps\\
& + \frac{1}{4} \partial_x^2 N_\eps \partial_x \Theta_\eps \big( \partial_x^2 \Theta_\eps \big)^2 - \frac{5}{12} N_\eps \partial_x N_\eps \big( \partial_x \Theta_\eps \big)^2 \partial_x^2 \Theta_\eps + \frac{1}{2} \partial_x^2 N_\eps \big( \partial_x \Theta_\eps \big)^2 \partial_x^3 \Theta_\eps -\frac{25}{144} \big( \partial_x N_\eps \big)^2 \big( \partial_x \Theta_\eps \big)^3\\
& + \frac{1}{216} N_\eps \big( \partial_x \Theta_\eps \big)^5 - \frac{1}{12} \partial_x N_\eps \big( \partial_x^2 \Theta_\eps \big)^3 - \frac{5 m_\eps}{72} \big( \partial_x \Theta_\eps \big)^3 \big( \partial_x^2 \Theta_\eps \big)^2 - \frac{5}{36 m_\eps} \big( \partial_x N_\eps \big)^3 \partial_x^2 \Theta_\eps\\
& + \frac{5}{72 m_\eps} \big( \partial_x N_\eps \big)^2 \big( \partial_x \Theta_\eps \big)^3 - \frac{1}{m_\eps} \partial_x N_\eps \partial_x^3 N_\eps \partial_x^3 \Theta_\eps - \frac{1}{2 m_\eps} \big( \partial_x^3 N_\eps \big)^2 \partial_x \Theta_\eps + \frac{5}{18 m_\eps^2} \big( \partial_x N_\eps \big)^3 \partial_x^2 \Theta_\eps \Big)\\
& + \frac{\eps^2}{4} \Big( \frac{245}{432 m_\eps^2} \big( \partial_x N_\eps \big)^4 \partial_x \Theta_\eps - \frac{21}{1296} N_\eps^2 \big( \partial_x \Theta_\eps \big)^5 -\frac{m_\eps}{1296} \big( \partial_x \Theta_\eps \big)^7 - \frac{35}{36 m_\eps} \big( \partial_x^2 N_\eps \big)^2 \big( \partial_x \Theta_\eps \big)^3\\
& - \frac{35}{6 m_\eps} \partial_x N_\eps \partial_x^2 N_\eps \big( \partial_x \Theta_\eps \big)^2 \partial_x^2 \Theta_\eps - \frac{35}{12 m_\eps} \big( \partial_x N_\eps \big)^2 \partial_x \Theta_\eps \big( \partial_x \Theta_\eps \big)^2 + \frac{7}{2 m_\eps^2} \partial_x N_\eps \big( \partial_x^2 N_\eps \big)^2 \partial_x^2 \Theta_\eps\\
& - \frac{7}{m_\eps^2} \big( \partial_x N_\eps \big)^2 \partial_x^2 N_\eps \partial_x^3 \Theta_\eps +\frac{7}{2 m_\eps^2} \big( \partial_x^2 N_\eps \big)^3 \partial_x \Theta_\eps - \frac{175}{216 m_\eps^3} \big( \partial_x N_\eps \big)^4 \partial_x \Theta_\eps - \frac{7}{108} \partial_x^2 N_\eps \big( \partial_x \Theta_\eps \big)^5 \Big)\\
& + \frac{\eps^4}{48} \Big( \frac{35}{m_\eps^3} \big( \partial_x N_\eps \big)^3 \partial_x^2 N_\eps \partial_x^2 \Theta_\eps - \frac{7}{72 m_\eps} \big( \partial_x N_\eps \big)^2 \big( \partial_x \Theta_\eps \big)^5 - \frac{25}{6 m_\eps^2} \big( \partial_x N_\eps \big)^3 \big( \partial_x \Theta_\eps \big)^2 \partial_x^2 \Theta_\eps\\
& - \frac{5}{18 m_\eps^2} \big( \partial_x N_\eps \big)^2 \partial_x^2 N_\eps \big( \partial_x \Theta_\eps \big)^3 + \frac{49}{2 m_\eps^3} \big( \partial_x N_\eps \big)^2 \big( \partial_x^2 N_\eps \big)^2 \partial_x \Theta_\eps \Big) + \frac{\eps^6}{768} \Big( \frac{5}{3 m_\eps^3} \big( \partial_x N_\eps \big)^4 \big( \partial_x \Theta_\eps \big)^3\\
& + \frac{252}{5 m_\eps^4} \big( \partial_x N_\eps \big)^5 \partial_x^2 \Theta_\eps - \frac{49 \eps^2}{10 m_\eps^5} \big( \partial_x N_\eps \big)^6 \partial_x \Theta_\eps \Big).
\end{split}
\end{equation}
\end{lemma}

\subsection{Relating the \eqref{GP} invariants to the \eqref{KdV} invariants}
\label{Kdvpitre}

Recall that the Korteweg-de Vries equation is integrable, and admits an infinite number of invariants (see \cite{GarKrMi1}). The first four invariants for \eqref{KdV} are given by
\begin{align}
\label{e0KdV}
E^{KdV}_0(v) \equiv & \frac{1}{2} \int_\R v^2,\\
\label{e1KdV}
E^{KdV}_1(v) \equiv & \frac{1}{2} \int_\R \big( \partial_x v \big)^2 - \frac{1}{6} \int_\R v^3,\\
\label{e2KdV}
E^{KdV}_2(v) \equiv & \frac{1}{2} \int_\R \big( \partial_x^2 v \big)^2 - \frac{5}{6} \int_\R v \big( \partial_x v \big)^2 + \frac{5}{72} \int_\R v^4,
\end{align}
and
\begin{equation}
\label{e3KdV}
E^{KdV}_3(v) \equiv \frac{1}{2} \int_\R \big( \partial_x^3 v \big)^2 - \frac{7}{6} \int_\R v \big( \partial_x^2 v \big)^2 + \frac{35}{36} \int_\R v^2 \big( \partial_x v \big)^2 - \frac{7}{216} \int_\R v^5.
\end{equation}
Notice that the invariants $E_k$ are bounded in terms of the $ H^k$-norm, since we have
\begin{equation}
\label{bornation1}
\big| E_k^{KdV}(v) \big| \leq K \big( \| v \|_{H^{k-1}(\R)} \big) \| v \|_{H^k(\R)}^2,
\end{equation}
where $K \big( \| v \|_{H^{k-1}(\R)} \big)$ is some constant depending only on the $H^{k-1}$-norm of $v$.

Another important observation concerning the \eqref{KdV} invariants is that, given any function $v \in H^k(\R)$, the $H^k$-norm of $v$ is controlled by the first $k^{th}$-invariants of \eqref{KdV}. This claim is straightforward for $k = 0$, whereas for $k \geq 1$, we have

\begin{lemma}
\label{Kdvcontrol}
Let $1 \leq k \leq 3$ be given. Given any function $v \in H^k(\R)$, there exists some positive constant $K = K \big( \| v \|_{H^{k-1}(\R)} \big)$, depending only on the $H^{k-1}$-norm of $v$, such that
\begin{equation}
\label{bornation2}
\| \partial_x^k v \|_{L^2(\R)}^2 \leq K \Big( \| v \|_{H^{k-1}(\R)}^2 + \big| E_k^{KdV}(v) \big| \Big).
\end{equation}
\end{lemma}

\begin{proof}
For $k = 2$ and $k = 3$, the proof of \eqref{bornation2} is a direct application of the Sobolev embedding theorem to formulae \eqref{e2KdV} and \eqref{e3KdV}. For $k = 1$, we have in view of \eqref{e1KdV},
$$\| \partial_x v \|_{L^2(\R)}^2 \leq \frac{1}{3} \| v \|_{L^2(\R)}^2 \| \partial_x v \|_{L^2(\R)} + 2 \big| E_1^{KdV}(v) \big|,$$
so that by the inequality $2 a b \leq a^2 + b^2$,
$$\| \partial_x v \|_{L^2(\R)}^2 \leq \frac{1}{9} \| v \|_{L^2(\R)}^4 + 4 \big| E_1^{KdV}(v) \big| \leq K \Big( \| v \|_{L^2(\R)}^2 + \big| E_1^{KdV}(v) \big| \Big),$$
where $K = \max \big\{ \frac{1}{9} \| v \|_{L^2(\R)}^2, 4 \big\}$.
\end{proof}

We complete the subsection showing that the \eqref{KdV} invariant $E_{k-1}^{KdV}$ is related to the \eqref{GP} invariant quantities $\boE_k \pm \sqrt{2} \boP_k$. For that purpose, assume that
$$N_\eps \to N_0 \ {\rm in} \ H^1(\R), \ {\rm and} \ \partial_x \Theta_\eps \to \partial _x \Theta_0 \ {\rm in} \ L^2(\R), \ {\rm as} \ \eps \to 0.$$
For $k = 1$, we notice in view of expansions \eqref{sirbu1} and \eqref{sirbu2}, that
\begin{equation}
\label{limit1}
\boE_1(N_\eps, \Theta_\eps) \pm \sqrt{2} \boP_1(N_\eps, \Theta_\eps) \to \frac{1}{8} \int_\R \big( N_0 \pm \partial_x \Theta_0 \big)^2 = E_0^{KdV} \Big( \frac{N_0 \pm \partial_x \Theta_0}{2} \Big), \ {\rm as} \ \eps \to 0.
\end{equation}
Similarly, it follows from Lemmas \ref{Esquale} and \ref{Psquale} that

\begin{prop}
\label{Limitep}
Let $1 \leq k \leq 4$ and $\Psi$ in $X^k(\R)$ which satisfies \eqref{toutpetit}. Denoting $N_\varepsilon$ and $\Theta_\varepsilon$ the variables defined by \eqref{slow-var}, and assuming that
$$N_\eps \to N_0 \ {\rm in} \ H^k(\R), \ {\rm and} \ \partial_x \Theta_\eps \to \partial_x \Theta_0 \ {\rm in} \ H^{k-1}(\R), \ {\rm as} \ \eps \to 0,$$
we have
$$\boE_k(N_\eps, \Theta_\eps) \pm \sqrt{2} \boP_k(N_\eps, \Theta_\eps) \to E_{k-1}^{KdV} \Big( \frac{N_0 \pm \partial_x \Theta_0}{2} \Big), \ {\rm as} \ \eps \to 0.$$
\end{prop}

\begin{remark}
We believe that Proposition \ref{Limitep} might be extended to higher order \eqref{GP} and \eqref{KdV} invariants, provided one was first able to compute some expressions for them.
\end{remark}

\begin{proof}
Combining the expansions of Lemmas \ref{Esquale} and \ref{Psquale} with \eqref{e1KdV}, \eqref{e2KdV} and \eqref{e3KdV}, and using the Sobolev embedding theorem, the proof reduces to a direct computation similar to the proof of \eqref{limit1}.
\end{proof}

\subsection{$H^k$-estimates for $N_\eps$ and $\partial_x \Theta_\eps$}
\label{Lecontrole}

In the same spirit as Lemma \ref{Kdvcontrol} which allows to bound the $H^k$-norms by the \eqref{KdV} invariants, we next show that the $H^k$-norms of $N_\eps$ and $\partial_x \Theta_\eps$ are controlled by the quantities $\boE_k(N_\eps, \Theta_\eps)$ in the limit $\eps \to 0$. More precisely, we have

\begin{lemma}
\label{Controlhk}
Let $1 \leq k \leq 4$ be given, and assume that there exists some positive constant $A$ such that
\begin{equation}
\label{hypoener}
\boE_j(N_\eps, \Theta_\eps) \leq A,
\end{equation}
for any $1 \leq j \leq k$. Then, there exists some positive numbers $\eps_A$ and $K_A$, possibly depending on $A$, such that
\begin{equation}
\label{bornitude}
\| N_\eps \|_{H^{k-1}(\R)} + \eps \| \partial_x^k N_\eps \|_{L^2(\R)} + \| \partial_x \Theta_\eps \|_{H^{k-1}(\R)} \leq K_A,
\end{equation}
for any $0 < \eps < \eps_A$.
\end{lemma}

\begin{remark}
We again believe that Lemma \ref{Controlhk} might be extended to higher order \eqref{GP} and \eqref{KdV} invariants, which will provide bounds for higher Sobolev norms of $N_\eps$ and $\partial_x \Theta_\eps$.
\end{remark}

\begin{proof}
We split the proof in four steps according to the value of $k$.

\begin{step}
\label{one-way}
$k = 1$.
\end{step}

In view of \eqref{sirbu1}, assumption \eqref{hypoener} may be written as
\begin{equation}
\label{hypoener1}
\int_\R \Big( m_\eps (\partial_x \Theta_\varepsilon)^2 + N_\varepsilon^2 + \frac{\varepsilon^2 (\partial_x N_\varepsilon)^2}{2 m_\eps} \Big) \leq 8 A,
\end{equation}
so that \eqref{bornitude} follows once some lower and upper uniform bounds on $m_\eps$ are established. Indeed, if we can choose $\eps_A$ so that
\begin{equation}
\label{petit}
\frac{1}{2} \leq m_\eps \leq 2,
\end{equation}
for any $0 < \eps < \eps_A$, then \eqref{bornitude} follows from \eqref{hypoener1} with $K_A = 24 \sqrt{A}$. Hence, the proof reduces to show \eqref{petit} for some suitable choice of $\eps_A$.

In order to prove \eqref{petit}, we apply the H\"older inequality and assumption \eqref{hypoener} to obtain
$$|\Psi(x) - \Psi(x_0)| \leq \sqrt{2} |x - x_0|^\frac{1}{2} E(\Psi)^\frac{1}{2} \leq \frac{\eps^\frac{3}{2}}{3} \boE_1(N_\eps, \Theta_\eps)^\frac{1}{2} \leq \frac{\sqrt{A}}{3} \eps^\frac{3}{2},$$
for any point $x_0 \in \R$ and any $x_0 - 1 \leq x \leq x_0 + 1$, so that
$$\big| 1 - |\Psi(x_0)| \big| - \frac{\sqrt{A}}{3} \eps^\frac{3}{2} \leq \big| 1 - |\Psi(x)| \big|.$$
Setting $\eps_A = (16 A)^{- \frac{1}{3}}$, and assuming by contradiction that \eqref{petit} does not hold at the point $x_0$, we obtain that
\begin{equation}
\label{hyporealestate}
\big| 1 - |\Psi(x_0)| \big| = \big| 1 - \sqrt{m_\eps(x_0)} \big| \geq 1 - \frac{1}{\sqrt{2}} \geq \frac{1}{12} = \frac{\sqrt{A}}{3} \eps_A^3 \geq \frac{\sqrt{A}}{3} \eps^3,
\end{equation}
for any $0 < \eps < \eps_A$, so that
$$\Big( \big| 1 - |\Psi(x_0)| \big| - \frac{\sqrt{A}}{3} \eps^\frac{3}{2} \Big)^2 \leq \int_{x_0 - 1}^{x_0 + 1} \big( 1 - |\Psi(x)|^2 \big)^2 dx \leq \frac{2 \eps^3}{9} \boE_1(N_\eps, \Theta_\eps) \leq \frac{2 A}{9} \eps^3,$$
and
$$\big| 1 - |\Psi(x_0)| \big| \leq \sqrt{A} \eps^\frac{3}{2} \leq \frac{1}{4}.$$
It follows that
$$\frac{9}{16} \leq m_\eps(x_0) = |\Psi(x_0)|^2 \leq \frac{25}{16},$$
which gives a contradiction with the fact that \eqref{petit} does not hold at the point $x_0$. This completes the proof of \eqref{petit}, and of Step \ref{one-way}.

\begin{step}
\label{two-ways}
$k = 2$.
\end{step}

Notice first that in view of the inductive nature of assumption \eqref{hypoener}, and of Step \ref{one-way}, we have already established \eqref{bornitude} for $k = 1$. Combining this estimate with the Sobolev embedding theorem, bounds \eqref{petit} and formulae \eqref{boE2} and \eqref{R2eps}, assumption \eqref{hypoener} may be written as
\begin{equation}
\begin{split}
\label{hypoener2}
& \int_\R \bigg( \big( \partial_x N_\eps \big)^2 + \Big( \partial_x^2 \Theta_\eps - \frac{\eps^2}{6 m_\eps} \partial_x N_\eps \partial_x \Theta_\eps \Big)^2 + \eps^2 \Big( \partial_x^2 N_\eps + \frac{m_\eps}{6} \big( \partial_x \Theta_\eps \big)^2 + \frac{\eps^2}{12 m_\eps} (\partial_x N_\eps)^2 \Big)^2 \bigg)\\
& \leq K \Big( 1 + \| N_\eps \|_{H^1(\R)} \big( \| N_\eps \|_{L^2(\R)}^2 + \| \partial_x \Theta_\eps \|_{L^2(\R)}^2 + \eps^2 \| \partial_x N_\eps \|_{L^2(\R)}^2 \big) \Big) \leq K_A \Big( 1 + \| N_\eps \|_{H^1(\R)} \Big).
\end{split}
\end{equation}
This first gives that
\begin{equation}
\label{unicredit1}
\int_\R \big( \partial_x N_\eps \big)^2 \leq K_A,
\end{equation}
so that by \eqref{hypoener2} and the Sobolev embedding theorem,
$$\| \partial_x^2 \Theta_\eps \|_{L^2(\R)} \leq K \Big( 1 + \eps^2 \| \partial_x N_\eps \partial_x \Theta_\eps \|_{L^2(\R)} \Big) \leq K_A \Big( 1 + \eps^2 \| \partial_x \Theta_\eps \|_{H^1(\R)} \Big).$$
Hence, we obtain
$$\int_\R \big( \partial_x^2 \Theta_\eps \big)^2 \leq K_A,$$
setting $\eps_A$ sufficiently small. In view of \eqref{hypoener2}, \eqref{unicredit1} and the Sobolev embedding theorem, it follows that
$$\eps \| \partial_x^2 N_\eps \|_{L^2(\R)} \leq K \Big( 1 + \eps \| \partial_x \Theta_\eps \|_{L^4(\R)}^2 + \eps^3 \| \partial_x N_\eps \|_{H^1(\R)} \| \partial_x N_\eps \|_{L^2(\R)} \Big) \leq K_A \Big( 1 + \eps^3 \| \partial_x N_\eps \|_{H^1(\R)} \Big),$$
which completes the proof of \eqref{bornitude} choosing $\eps_A$ sufficiently small.

\begin{step}
\label{freeway}
$k = 3$.
\end{step}

Notice again that in view of the inductive nature of assumption \eqref{hypoener}, and of Step \ref{two-ways}, we have already established \eqref{bornitude} for $k = 2$. Combining this estimate with the Sobolev embedding theorem, bounds \eqref{petit} and formulae \eqref{boE3} and \eqref{R3eps}, assumption \eqref{hypoener} may be written as
\begin{align*}
& \int_\R \bigg( \big( \partial_x^2 N_\eps \big)^2 + \Big( \partial_x^3 \Theta_\eps - \frac{\eps^2}{72} \big( \partial_x \Theta_\eps \big)^3 - \frac{\eps^2 \partial_x N_\eps \partial_x^2 \Theta_\eps}{4 m_\eps} - \frac{\eps^2 \partial_x^2 N_\eps \partial_x \Theta_\eps}{4 m_\eps} - \frac{\eps^4 (\partial_x N_\eps)^2 \partial_x \Theta_\eps}{48 m_\eps^2} \Big)^2\\
& + \eps^2 \Big( \partial_x^3 N_\eps - \frac{\eps^2}{24} \partial_x N_\eps \big( \partial_x \Theta_\eps \big)^2 + \frac{m_\eps}{2} \partial_x \Theta_\eps \partial_x^2 \Theta_\eps + \frac{\eps^2}{4 m_\eps} \partial_x N_\eps \partial_x^2 N_\eps + \frac{\eps^4}{48 m_\eps^2} \big( \partial_x N_\eps \big)^3 \Big)^2 \bigg)\\
& \leq K_A \Big( 1 + \| \partial_x N_\eps \|_{H^1(\R)} + \| \partial_x^2 \Theta_\eps \|_{H^1(\R)} + \eps \| \partial_x^3 N_\eps \|_{L^2(\R)} \Big).
\end{align*}
Invoking once again estimates \eqref{bornitude} (for $k = 2$) and \eqref{petit} to bound the remainder terms in the above integral, we are led to
$$\int_\R \bigg( \big( \partial_x^2 N_\eps \big)^2 + \big( \partial_x^3 \Theta_\eps \big)^2 + \eps^2 \big( \partial_x^3 N_\eps \big)^2 \Big) \leq K_A \Big( 1 + \| \partial_x N_\eps \|_{H^1(\R)} + \| \partial_x^2 \Theta_\eps \|_{H^1(\R)} + \eps \| \partial_x^3 N_\eps \|_{L^2(\R)} \Big),$$
which provides the proof of Step \ref{freeway}.

\begin{step}
\label{forwhat}
$k = 4$.
\end{step}

Notice once last time that, in view of the inductive nature of assumption \eqref{hypoener}, and of Step \ref{two-ways}, we have already established \eqref{bornitude} for $k = 3$. Combining this estimate with the Sobolev embedding theorem, bounds \eqref{petit} and formulae \eqref{boE4} and \eqref{R4eps}, assumption \eqref{hypoener} may be written as
\begin{align*}
& \int_\R \bigg( \big( \partial_x^3 N_\eps \big)^2 + \Big( \partial_x^4 \Theta_\eps - \frac{\eps^2}{2 m_\eps} \partial_x^2 N_\eps \partial_x^2 \Theta_\eps - \frac{\eps^2}{3 m_\eps} \partial_x N_\eps \partial_x^3 \Theta_\eps - \frac{\eps^2}{3 m_\eps} \partial_x^3 N_\eps \partial_x \Theta_\eps\\
& - \frac{\eps^2}{12} \big( \partial_x \Theta_\eps \big)^2 \partial_x^2 \Theta_\eps - \frac{\eps^4}{24 m_\eps^2} \big( \partial_x N_\eps \big)^2 \partial_x^2 \Theta_\eps - \frac{\eps^4}{12 m^2} \partial_x N_\eps \partial_x^2 N_\eps \partial_x \Theta_\eps + \frac{\eps^4}{216 m} \partial_x N_\eps \big( \partial_x \Theta_\eps \big)^3\\
& - \frac{\eps^6}{144 m_\eps^3} \big( \partial_x N_\eps \big)^3 \partial_x \Theta_\eps \Big)^2 + \eps^2 \int_\R \Big( \partial_x^4 N_\eps + \frac{m_\eps}{2} \big( \partial_x^2 \Theta_\eps \big)^2 + \frac{2}{3} m_\eps \partial_x \Theta_\eps \partial_x^3 \Theta_\eps + \frac{\eps^2}{3 m_\eps} \partial_x N_\eps \partial_x^3 N_\eps\\
& + \frac{\eps^2}{4 m_\eps} \big( \partial_x^2 N_\eps \big)^2 - \frac{\eps^2}{12} \partial_x^2 N_\eps \big( \partial_x \Theta_\eps \big)^2 - \frac{\eps^2}{6} \partial_x N_\eps \partial_x \Theta_\eps \partial_x^2 \Theta_\eps - \frac{\eps^2}{432} m_\eps \big( \partial_x \Theta_\eps \big)^4\\
& - \frac{\eps^4}{144 m_\eps} \big( \partial_x N_\eps \big)^2 \big( \partial_x \Theta_\eps \big)^2 + \frac{\eps^4}{8 m_\eps^2} \big( \partial_x N_\eps \big)^2 \partial_x^2 N_\eps + \frac{5 \eps^6}{576 m_\eps^3} \big( \partial_x N_\eps \big)^4 \Big)^2\\
& \leq K_A \bigg( 1 + \| \partial_x^2 N_\eps \|_{H^1(\R)} + \| \partial_x^3 \Theta_\eps \|_{H^1(\R)} \bigg),
\end{align*}
so that we similarly obtain
$$\int_\R \bigg( \big( \partial_x^3 N_\eps \big)^2 + \big( \partial_x^4 \Theta_\eps \big)^2 \bigg) \leq K_A,$$
then, combining with the Sobolev embedding theorem, we also have
$$\eps \| \partial_x^4 N_\eps \|_{L^2(\R)} \leq K_A.$$
This completes the proofs of Step \ref{forwhat} and Lemma \ref{Controlhk}.
\end{proof}

An important consequence of Lemma \ref{Controlhk} which refines the result of Lemma \ref{Limitep} is

\begin{prop}
\label{Controluv}
Let $1 \leq k \leq 3$. Given some positive constant $A$, consider some functions $N_\eps$ and $\partial_x \Theta_\eps$ which satisfy \eqref{hypoener} for any $1 \leq j \leq k + 1$. Then, there exists some positive numbers $\eps_A$ and $K_A$, possibly depending on $A$, such that
\begin{equation}
\label{bornage}
\Big| \boE_k(N_\eps, \Theta_\eps) \pm \sqrt{2} \boP_k(N_\eps, \Theta_\eps) - E_{k-1}^{KdV} \Big( \frac{N_\eps \pm \partial_x \Theta_\eps}{2} \Big) \Big| \leq K_A \eps^2.
\end{equation}
for any $0 < \eps < \eps_A$.
\end{prop}

\begin{remark}
Similarly to Lemma \ref{Controlhk}, we believe that Proposition \ref{Controluv} might be extended to higher order \eqref{GP} and \eqref{KdV} invariants.
\end{remark}

\begin{proof}
Let $k = 1$. In view of \eqref{sirbu1}, \eqref{sirbu2} and \eqref{e0KdV}, we have
$$\boE_1(N_\eps, \Theta_\eps) \pm \sqrt{2} \boP_1(N_\eps, \Theta_\eps) - E_0^{KdV} \Big( \frac{N_\eps \pm \partial_x \Theta_\eps}{2} \Big) = \frac{\eps^2}{2} \int_\R \bigg( \frac{(\partial_x N_\eps)^2}{1 - \frac{\eps^2}{6} N_\eps} - \frac{1}{3} N_\eps \big( \partial_x \Theta_\eps \big)^2 \bigg).$$
Inequality \eqref{bornage} follows for $\eps_A$ sufficiently small, invoking \eqref{bornitude} (for $k = 2$) and \eqref{petit}.

For $k = 2$, we deduce from \eqref{boE2}, \eqref{R2eps} and \eqref{e1KdV} that
\begin{align*}
& \boE_2(N_\eps, \Theta_\eps) \pm \sqrt{2} \boP_2(N_\eps, \Theta_\eps) - E_1^{KdV} \Big( \frac{N_\eps \pm \partial_x \Theta_\eps}{2} \Big)\\
= \frac{\eps^2}{8} \int_\R \bigg( \frac{1}{12} N_\eps^2 \big( \partial_x \Theta_\eps \big)^2 & - \frac{1}{3} \partial_x N_\eps \partial_x \Theta_\eps \partial_x^2 \Theta_\eps - \frac{1}{6} N_\eps \big( \partial_x^2 \Theta_\eps \big)^2 \pm \frac{1}{36} N_\eps \big( \partial_x \Theta_\eps \big)^3 - \frac{N_\eps (\partial_x N_\eps)^2}{4 m_\eps}\\
\mp \frac{(\partial_x N_\eps)^2 \partial_x \Theta_\eps}{4 m_\eps} + & \frac{\eps^2 (\partial_x N_\eps)^2 (\partial_x \Theta_\eps)^2}{36 m_\eps} + \frac{1}{2 m_\eps} \Big( \partial_x^2 N_\eps + \frac{m_\eps}{6} \big( \partial_x \Theta_\eps \big)^2 + \frac{\eps^2}{12 m_\eps} \big( \partial_x N_\eps \big)^2 \Big)^2 \bigg),
\end{align*}
so that \eqref{bornage} follows again from \eqref{bornitude} (for $k = 3$) and \eqref{petit}.

Similarly, the proof of \eqref{bornage} for $k = 3$ reduces to estimate the remainder terms in \eqref{boE3} and \eqref{R3eps} using \eqref{bornitude} (for $k = 4$) and \eqref{petit}.
\end{proof}

\section{Time-independent estimates}
\label{Notime}

In this section, we use the above conservation laws to derive time-independent estimates of the functions $U_\eps$ and $V_\eps$, together with the consistency of the solutions to \eqref{GP} with the \eqref{KdV} equation in the limit $\eps \to 0$. This yields the proofs of Proposition \ref{H3-control} and Theorem \ref{Bobby}.

\subsection{Proof of Proposition \ref{H3-control}}
\label{Cestpresquefini}

Given any functions $N_\eps^0$ and $\Theta_\eps^0$ such that \eqref{grinzing1} holds, it follows from the formulae of Lemmas \ref{Esquale} and \ref{Psquale} that there exists some positive constant $A_0$, which does not depend on $\varepsilon$, such that
\begin{equation}
\label{hypoenerzero}
\boE_k(N_\eps^0, \Theta_\eps^0) \leq A_0,
\end{equation}
for any $1 \leq k \leq 4$. In view of Theorem \ref{Invconserved} and definition \eqref{slowEk}, we deduce that the solution $(N_\eps(\cdot, \tau), \Theta_\eps(\cdot, \tau))$ to system \eqref{slow1-0}-\eqref{slow2-0} with initial datum $(N_\eps^0, \Theta_\eps^0)$ satisfies
$$\boE_k(N_\eps(\cdot, \tau), \Theta_\eps(\cdot, \tau)) \leq A_0,$$
for any time $\tau \in \R$. In particular, inequality \eqref{grinzing1bis} is a direct consequence of Lemma \ref{Controlhk}, whereas in view of Proposition \ref{Controluv}, we have
$$\Big| \boE_k(N_\eps(\cdot, \tau), \Theta_\eps(\cdot, \tau)) \pm \sqrt{2} \boP_k(N_\eps(\cdot, \tau), \Theta_\eps(\cdot, \tau)) - E_{k-1}^{KdV} \Big( \frac{N_\eps(\cdot, \tau) \pm \partial_x \Theta_\eps(\cdot, \tau)}{2} \Big) \Big| \leq K_{A_0} \eps^2,$$
for any time $\tau \in \R$. Using again the conservation of $E_k$ and $p_k$ provided by Theorem \ref{Invconserved} and Corollary \ref{Conspk}, and definitions \eqref{slowEk} and \eqref{slowpk}, we are led to
$$\Big| \boE_k(N_\eps^0, \Theta_\eps^0) \pm \sqrt{2} \boP_k(N_\eps^0, \Theta_\eps^0) - E_{k-1}^{KdV} \Big( \frac{N_\eps(\cdot, \tau) \pm \partial_x \Theta_\eps(\cdot, \tau)}{2} \Big) \Big| \leq K_{A_0} \eps^2.$$
Invoking \eqref{hypoenerzero}, we apply once more Proposition \ref{Controluv} to obtain
\begin{equation}
\label{lloyds}
\Big| E_{k-1}^{KdV} \Big( \frac{N_\eps^0 \pm \partial_x \Theta_\eps^0}{2} \Big) - E_{k-1}^{KdV} \Big( \frac{N_\eps(\cdot, \tau) \pm \partial_x \Theta_\eps(\cdot, \tau)}{2} \Big) \Big| \leq K_{A_0} \eps^2.
\end{equation}
For $k = 1$, we then deduce from \eqref{e2KdV} that
\begin{equation}
\label{abbey}
\| N_\eps(\cdot, \tau) \pm \partial_x \Theta_\eps(\cdot, \tau) \|_{L^2(\R)} \leq \| N_\eps^0 \pm \partial_x \Theta_\eps^0 \|_{L^2(\R)} + K_{A_0} \eps,
\end{equation}
so that in particular, we have by \eqref{grinzing1} for $\eps$ sufficiently small,
$$\| N_\eps(\cdot, \tau) \pm \partial_x \Theta_\eps(\cdot, \tau) \|_{L^2(\R)} \leq K_{A_0},$$
where $K_{A_0}$ denotes some further constant depending only on $A_0$. Hence, for $k = 2$, we may write using \eqref{bornation2}, that
$$\big\| \partial_x N_\eps(\cdot, \tau) \pm \partial_x^2 \Theta_\eps(\cdot, \tau) \big\|_{L^2} \leq K_{A_0} \Big( \Big| E_1^{KdV} \Big( \frac{N_\eps(\cdot, \tau) \pm \partial_x \Theta_\eps(\cdot, \tau)}{2} \Big) \Big| + \| N_\eps(\cdot, \tau) \pm \partial_x \Theta_\eps(\cdot, \tau) \|_{L^2} \Big),$$
so that by \eqref{bornation1}, \eqref{lloyds} and \eqref{abbey},
$$\| \partial_x N_\eps(\cdot, \tau) \pm \partial_x^2 \Theta_\eps(\cdot, \tau) \|_{L^2(\R)} \leq K_{A_0} \Big( \| N_\eps^0 \pm \partial_x \Theta_\eps^0 \|_{H^1(\R)} + \eps \Big).$$
Using repetitively this argument to estimate the $L^2$-norms of the functions $\partial_x^2 N_\eps(\cdot, \tau) \pm \partial_x^3 \Theta_\eps(\cdot, \tau)$ and $\partial_x^3 N_\eps(\cdot, \tau) \pm \partial_x^4 \Theta_\eps(\cdot, \tau)$, we are led to \eqref{dobling1bis}, which completes the proof of Proposition \ref{H3-control}.

\subsection{Proof of Theorem \ref{Bobby}}
\label{bobbydone}

Theorem \ref{Bobby} is a consequence of Proposition \ref{H3-control}. Applying estimates \eqref{dobling1bis} to the right-hand side of \eqref{slow1}, together with the Sobolev embedding theorem, we obtain estimate \eqref{jerrard}.

\section{Energy methods}
\label{Expansion}

This section is devoted to the proofs of Theorems \ref{cochon} and \ref{H3-controlbis}, which both rely on applying standard energy methods to equations \eqref{slow1} and \eqref{slow2}.

\subsection{Proof of Theorem \ref{H3-controlbis}}
\label{Veve}

In order to estimate the $L^2$-norm of $V_\eps(\cdot, \tau)$, we multiply equation \eqref{slow2} by $V_\eps(\cdot, \tau)$ and integrate by parts. In order to simplify the presentation, we recast equation \eqref{slow2} as
\begin{equation}
\label{soleillevant}
\partial_\tau V_\varepsilon + \frac{8}{\varepsilon^2} \partial_x V_\varepsilon = \frac{1}{2} \partial_x (V_\varepsilon^2) + \partial_x f_\eps + \eps^2 R_\eps,
\end{equation}
where
$$f_\eps = \partial^2_x N_\varepsilon - \frac{1}{6} U_\varepsilon^2 - \frac{1}{3} U_\varepsilon V_\varepsilon,$$
and $R_\eps$ is defined in \eqref{grouin}. We are led to
$$\partial_\tau \bigg( \int_\R V_\eps(\cdot, \tau)^2 \bigg) = -2 \int_\R f_\eps \partial_x V_\eps(\cdot, \tau) + 2 \varepsilon^2 \int_\R R_\varepsilon(\cdot, \tau) V_\eps(\cdot, \tau).$$
We now integrate with respect to the time variable to obtain
\begin{equation}
\label{marseille}
\int_\R \big( V_\eps(\cdot, \tau) \big)^2 = \int_\R \big( V_\eps^0 \big)^2 -2 \int_0^\tau \int_\R f_\eps \partial_x V_\eps + 2 \varepsilon^2 \int_0^\tau \int_\R R_\varepsilon V_\eps.
\end{equation}
Combining inequalities \eqref{grinzing1bis} with definition \eqref{grouin} and bound \eqref{borninf} and using the Sobolev embedding theorem, we next have
\begin{equation}
\label{aubagne}
\| U_\eps(\cdot, \tau) \|_{H^3(\R)} + \| V_\eps(\cdot, \tau) \|_{H^3(\R)} + \| R_\eps(\cdot, \tau) \|_{L^2(\R)}+ \| f_\eps(\cdot, \tau) \|_{H^1(\R)} \leq K,
\end{equation}
for any $\tau \in \R$ and some positive constant $K$ not depending on $\eps$. In particular,
\begin{equation}
\label{pinard}
\bigg| 2 \eps^2 \int_0^\tau \int_\R R_\eps V_\eps \bigg| \leq C \eps^2 \bigg| \int_0^\tau \| V_\eps(\cdot, s) \|_{L^2(\R)} ds \bigg|,
\end{equation}
where $C = C(K)$ does not depend on $\eps$. In order to bound the second term in the right-hand side of \eqref{marseille}, we replace the quantity $\partial_x V_\eps$ in \eqref{marseille} according to \eqref{soleillevant}, so that
\begin{align*}
\int_0^\tau \int_\R f_\eps \partial_x V_\eps & = \frac{\eps^2}{8} \int_0^\tau \int_\R f_\eps \Big( - \partial_\tau V_\eps + \frac{1}{2} \partial_x (V_\eps^2) + \partial_x f_\eps +\eps^2 R_\eps \Big)\\
& \equiv J_1 + J_2 + J_3 + J_4.
\end{align*}
We bound each of the terms $J_k$ separately. First note that the integrand being a differential, $J_3 = 0$. Next, it follows from \eqref{aubagne} that
\begin{equation}
\label{lia}
|J_4| \leq C \eps^4 \tau.
\end{equation}
Concerning $J_2$, we have
\begin{equation}
\label{stan}
|J_2| = \bigg| \frac{\eps^2}{16} \int_0^\tau \int_\R f_\eps \partial_x (V_\eps^2) \bigg| = \bigg| \frac{\eps^2}{16} \int_0^\tau \int_\R \partial_x f_\eps V_\eps^2 \bigg| \leq C \eps^2 \bigg| \int_0^\tau \| V_\eps(\cdot, s) \|_{L^2(\R)} ds \bigg|.
\end{equation}
For $J_1$, we perform an integration by parts with respect to the time variable, so that
\begin{equation}
\label{persil}
J_1 = \frac{\eps^2}{8} \int_0^\tau \int_\R \partial_\tau f_\eps V_\eps - \frac{\eps^2}{8} \bigg[ \int_\R f_\eps V_\eps \bigg]^\tau_0.
\end{equation}
Note that by \eqref{slow1-0}, \eqref{slow1}, \eqref{grinzing1bis} and \eqref{aubagne},
\begin{align*}
\partial_\tau f_\eps & = \partial_x^2 \partial_\tau N_\eps - \frac{1}{3} U_\eps \partial_\tau U_\eps - \frac{1}{3} \partial_\tau U_\eps V_\eps - \frac{1}{3} U_\eps \partial_\tau V_\eps\\
& = - \frac{4}{\eps^2} \partial^3_x V_\eps - \frac{1}{3} U_\eps \partial_\tau V_\eps + \boO(1)
\end{align*}
uniformly in $L^2(\R)$, so that
\begin{equation}
\label{ariel}
\bigg| \frac{\eps^2}{8} \int_0^\tau \int_\R \partial_\tau f_\eps V_\eps \bigg| \leq \bigg| \frac{\eps^2}{48} \int_0^\tau \int_\R U_\eps \partial_\tau(V_\eps)^2 \bigg| + C \eps^2 \bigg| \int_0^\tau \| V_\eps(\cdot, s) \|_{L^2(\R)} ds \bigg|.
\end{equation}
A further integration by parts in time leads to
\begin{equation}
\begin{split}
\label{omo}
\frac{\eps^2}{48} \int_0^\tau \int_\R U_\eps \partial_\tau (V_\eps)^2 = - \frac{\eps^2}{48} \int_0^\tau \int_\R (\partial_\tau U_\eps) V_\eps^2 + \frac{\eps^2}{48} \bigg[ \int_\R U_\eps V_\eps^2 \bigg]^\tau_0,
\end{split}
\end{equation}
and since $\partial_\tau U_\eps$ is uniformly bounded in $L^2(\R)$ by \eqref{slow1}, \eqref{grinzing1bis} and \eqref{aubagne}, we obtain, combining \eqref{persil}, \eqref{ariel} and \eqref{omo},
\begin{equation}
\label{filou}
J_1 \leq C \eps^2 \bigg( \| V_\eps(\cdot, 0) \|_{L^2(\R)} + \| V_\eps(\cdot, \tau) \|_{L^2(\R)} + \bigg| \int_0^\tau \| V_\eps(\cdot, s) \|_{L^2(\R)} ds \bigg| \bigg).
\end{equation}
Finally, combining \eqref{marseille}, \eqref{pinard}, \eqref{lia}, \eqref{stan} and \eqref{filou}, we obtain
$$\| V_\eps(\cdot, \tau) \|_{L^2}^2 \leq \| V_\eps(\cdot, 0) \|_{L^2}^2 + C \eps^2 \bigg( \eps^2 \tau + \| V_\eps(\cdot, 0) \|_{L^2} + \| V_\eps(\cdot, \tau) \|_{L^2} + \bigg| \int_0^\tau \| V_\eps(\cdot, s) \|_{L^2} ds \bigg| \bigg).$$
The proof of Theorem \ref{H3-controlbis} then follows by the Gronwall lemma.

\subsection{Proof of Theorem \ref{cochon}}
\label{Uhu}

We first recall the equation \eqref{slow1} satisfied by $U_\eps$, namely
$$\partial_\tau U_\varepsilon + \partial_x^3 U_\varepsilon + U_\varepsilon \partial_x U_\varepsilon = - \partial_x^3 V_\varepsilon + \frac{1}{3} \partial_x \Big( U_\varepsilon V_\varepsilon + \frac{V_\varepsilon^2}{2} \Big) - \varepsilon^2 R_{\varepsilon},$$
and take the difference with the \eqref{KdV} equation
$$\partial_\tau \boN_\eps + \partial_x^3 \boN_\eps + \boN_\eps \partial_x \boN_\eps = 0,$$
so that $Z_\eps \equiv U_\eps - \boN_\eps$ satisfies the equation
\begin{equation}
\label{dede}
\partial_\tau Z_\eps + \partial_x^3 Z_\eps + Z_\eps \partial_x U_\eps + \boN_\eps \partial_x Z_\eps = - \partial_x^3 V_\varepsilon + \frac{1}{3} \partial_x \Big( U_\varepsilon V_\varepsilon + \frac{V_\varepsilon^2}{2} \Big) - \varepsilon^2 R_{\varepsilon}.
\end{equation}
We multiply \eqref{dede} by $Z_\eps$, integrate on $\R$ and perform an integration by parts to obtain
\begin{align*}
& \partial_\tau \| Z_\eps \|_{L^2(\R)}^2 \leq K \big( \| \partial_x U_\eps \|_{L^\infty(\R)} + \| \partial_x \boN_\eps \|_{L^\infty(\R)} \big) \| Z_\eps \|^2_{L^2(\R)}\\
+ K & \| Z_\eps \|_{L^2(\R)} \Big( \| V_\eps \|_{H^3(\R)} + \| V_\eps \|_{L^2(\R)} \big( \| U_\eps \|_{H^1(\R)} + \| V_\eps \|_{H^1(\R)} \big) + \eps^2 \| R_\eps \|_{L^2(\R)} \Big).
\end{align*}
Using bounds \eqref{aubagne} for $U_\eps$, $V_\eps$ and $R_\eps$, and the bound of $\boN_\eps$ in $H^3(\R)$ which follows from the integrability theory of \eqref{KdV}, we are led to
$$\partial_\tau \| Z_\eps \|_{L^2(\R)}^2 \leq K \| Z_\eps \|_{L^2(\R)}^2 + K \| Z_\eps \|_{L^2(\R)} \big( \eps^2 + \| V_\eps \|_{H^3(\R)} \big).$$
Finally, we invoke Proposition \ref{H3-control} to assert
$$\partial_\tau \| Z_\eps \|_{L^2(\R)}^2 \leq K \| Z_\eps \|_{L^2(\R)}^2 + K \| Z_\eps \|_{L^2(\R)} \big( \eps + \| V_\eps(\cdot, 0) \|_{H^3(\R)} \big),$$
so that by the Gronwall lemma,
\begin{equation}
\label{mittal}
\| Z_\eps(\cdot, \tau) \|_{L^2(\R)} \leq \| Z_\eps(\cdot, 0) \|_{L^2(\R)} + K \big( \eps + \| V_\eps(\cdot, 0) \|_{H^3(\R)} \big) \exp(K \tau).
\end{equation}
On the other hand, at time $\tau = 0$, since $\boN_\eps(\cdot, 0) = N_\eps(\cdot, 0)$, we have
\begin{equation}
\label{arcelor}
\| Z_\eps(\cdot, 0) \|_{L^2(\R)} = \| U_\eps(\cdot,0) - \boN_\eps(\cdot,0) \|_{L^2(\R)}= \| V_\eps(\cdot, 0)\|_{L^2(\R)},
\end{equation}
whereas at positive time, by definition of $V_\eps$, we have
\begin{equation}
\begin{split}
\label{cigale}
\| N_\eps(\cdot, \tau) - \boN_\eps(\cdot, \tau) \|_{L^2(\R)} & \leq \| Z_\eps(\cdot, \tau) \|_{L^2(\R)} + \| V_\eps(\cdot, \tau) \|_{L^2(\R)}\\
& \leq \| Z_\eps(\cdot, \tau) \|_{L^2(\R)} + \| V_\eps(\cdot, 0) \|_{L^2(\R)} + K \eps^2 |\tau|,
\end{split}
\end{equation}
where we have used Theorem \ref{H3-controlbis}. The conclusion for $N_\eps - \boN_\eps$ then follows from \eqref{mittal}, \eqref{arcelor} and \eqref{cigale}. The proof is similar for $\partial_x \Theta_\eps - \boM_\eps$ considering the function $Y_\eps \equiv U_\eps - \boM_\eps$ instead of $Z_\eps$, so that we omit it.

\bibliographystyle{plain}
\bibliography{Bibliogr}

\end{document}